\documentclass[12pt]{article}

\usepackage[utf8]{inputenc}
\usepackage[english]{babel}
\usepackage{amsthm}
\usepackage{amsmath}
\usepackage{amsfonts}
\usepackage{tikz}
\usepackage{appendix}
\usepackage[normalem]{ulem}
\usepackage{mathtools}
\usepackage{stmaryrd}
\usepackage[noadjust]{cite}
\usepackage{mathabx}

\usepackage{caption}
\DeclareCaptionLabelSeparator{none}{ }
\usepackage[left=2.5cm, right=2.5cm, top=2.5cm, bottom=2.5cm]{geometry}

\usepackage[inline, shortlabels]{enumitem}

\usepackage[hidelinks, bookmarks, bookmarksnumbered, pdfstartview={XYZ null null 1.00}]{hyperref}

\newtheorem{theorem}{Theorem}[section]
\newtheorem{lemma}[theorem]{Lemma}
\newtheorem{proposition}[theorem]{Proposition}
\newtheorem{corollary}[theorem]{Corollary}
\newtheorem{conjecture}[theorem]{Conjecture}
\theoremstyle{definition}
\newtheorem{definition}[theorem]{Definition}
\newtheorem{remark}[theorem]{Remark}

\DeclareMathOperator{\rank}{rank}
\DeclareMathOperator{\Ran}{Ran}
\DeclareMathOperator{\supp}{supp}

\DeclarePairedDelimiter{\abs}{\lvert}{\rvert}
\DeclarePairedDelimiter{\scalprod}{\langle}{\rangle}

\newcommand{\vertiii}[1]{{\left\vert\kern-0.25ex\left\vert\kern-0.25ex\left\vert #1 
        \right\vert\kern-0.25ex\right\vert\kern-0.25ex\right\vert}}

\newcommand{\vertii}[1]{{\left\vert\kern-0.25ex\left\vert #1 
        \right\vert\kern-0.25ex\right\vert}}

\newcommand{\suchthat}{\ifnum\currentgrouptype=16 \mathrel{}\middle|\mathrel{}\else\mid\fi}

\theoremstyle{definition}

\begin{document}

\title{Hautus--Yamamoto criteria for approximate and exact controllability of linear difference delay equations}

\author{Yacine Chitour\footnotemark[1] \and Sébastien Fueyo\footnotemark[2] \and Guilherme Mazanti\footnotemark[3] \and Mario Sigalotti\footnotemark[2]}

{\renewcommand{\thefootnote}{\fnsymbol{footnote}}
\footnotetext[1]{Universit\'e Paris-Saclay, CNRS, CentraleSup\'elec, Laboratoire des signaux et syst\`emes, 91190, Gif-sur-Yvette, France.}
\footnotetext[2]{Sorbonne Universit\'e, Inria, CNRS, Laboratoire Jacques-Louis Lions (LJLL), F-75005 Paris, France.}
\footnotetext[3]{Universit\'e Paris-Saclay, CNRS, CentraleSup\'elec, Inria, Laboratoire des signaux et syst\`emes, 91190, Gif-sur-Yvette, France. G.M.\ was partially supported by ANR PIA funding: ANR-20-IDEES-0002.}
}

\maketitle

\begin{abstract}
The paper deals with the controllability of finite-dimensional linear difference delay equations, i.e., dynamics for which the state at a given time $t$ is obtained as a linear combination of the control evaluated at time $t$ and of the state evaluated at finitely many previous  instants of time $t-\Lambda_1,\dots,t-\Lambda_N$. Based on the realization 
theory developed by Y.~Yamamoto for general infinite-dimensional dynamical systems, we obtain necessary and sufficient conditions, expressed in the frequency domain, for the approximate controllability in finite time in $L^q$ spaces, $q \in [1, +\infty)$. We also provide a necessary condition for $L^1$ exact controllability, which can be seen as the closure of the $L^1$ approximate controllability criterion. Furthermore, we provide an explicit upper bound on the minimal times of approximate and exact controllability, given by $d\max\{\Lambda_1,\dots,\Lambda_N\}$, where $d$ is the dimension of the state space.
\end{abstract}

{\small

\medskip

}

\section{Introduction}

The present paper deals with the approximate and exact controllability of linear difference 
delay equations of the form
\begin{equation}
\label{system_lin_formel2}
 x(t)=\sum_{j=1}^NA_jx(t-\Lambda_j)+Bu(t) , \qquad t \ge 0,
\end{equation}
where, given three positive integers $d$, $m$, and $N$, $A_1,\dotsc,A_N$ are fixed $d\times d$ 
matrices with real entries, the state $x$ and the control $u$ belong 
to $\mathbb{R}^d$ and $\mathbb{R}^m$ respectively, and $B$ is a fixed $d \times m$ matrix 
with real entries. Without loss of generality, the delays $\Lambda_1, \dotsc, \Lambda_N$ 
are positive real numbers so that $\Lambda_1< \dotsb <\Lambda_N$. 

One of the major interests in the study of difference delay equations is that some 1D hyperbolic partial differential equations (PDEs) and, more precisely, systems of linear conservation laws
can be put in the form~\eqref{system_lin_formel2} through the method of characteristics \cite{baratchart,bastin2016stability,Chitour2016Stability,CoNg}, yielding a system under the form \eqref{system_lin_formel2} with a specific structure for the matrices $A_1, \dotsc, A_N$ (which are, in particular, all of rank $1$). Equation~\eqref{system_lin_formel2} has mostly been analysed from a stability viewpoint, in the case where there is no open-loop control $u$ (or, equivalently, when $B = 0$). Necessary and sufficient criteria have been obtained for the exponential stability of the origin of the system \cite{Avellar,Chitour2016Stability,Hale,Henry}.

The main efforts to study the controllability properties of delay systems have been made on neutral differential delay systems
\begin{equation}
\label{system_neutral}
\frac{d}{dt} \left( x(t)-\sum_{j=1}^NA_jx(t-\Lambda_j) \right)=\sum_{j=0}^N\widetilde{A}_jx(t-\Lambda_j)+B u(t) , \qquad t \ge 0,
\end{equation}
where $\widetilde{A}_0,\dots,\widetilde{A}_N$ are $d \times d$ real matrices and $\Lambda_0=0$. Due to the infinite-dimensional nature of neutral functional differential equations and difference delay equations, several notions of controllability arise, such as approximate, exact, or relative controllability. Such notions of controllability can be particularized according to whether one requires or not that controllability occurs in a uniform time $T$ (see Definitions~\ref{def:Lq-cont} and \ref{def:contr-from-origin}).

The mostly investigated controllability notion is that of approximate controllability (with no finite upper bound on the controllability time), usually in a control space made of square integrable functions and a state space equal to the Sobolev space $W^{1,2}([-\Lambda_N,0],\mathbb{R}^d)$. Usual tools in this context are semigroups properties and Laplace transforms, which lead to Hautus-type conditions for controllability, i.e., rank conditions on certain matrix-valued holomorphic functions. When the system is \emph{retarded}, i.e.,  $A_j = 0$ for every  $j$, Manitius~\cite{Manitius} gave a necessary and sufficient 
condition for the approximate controllability of System~\eqref{system_neutral}. This result has 
been extended by Jacob \emph{et al.}~\cite{Jacobs} and O'Connor \emph{et al.}~\cite{connor} to the 
neutral case in the single-delay case $N=1$. Salamon~\cite{salamon1984control} gave dual formulations of exact controllability (which can be seen as observability inequalities) as well as some 
insights for the general case for small solutions. Finally, 
Yamamoto~\cite{yamamoto1989reachability}, through the infinite realization theory developed 
in~\cite{yamamoto1981realization}, expanded the result on approximate controllabiity to the general neutral case 
with an arbitrary number of delays.

Surprisingly, the controllability of the difference delay system \eqref{system_lin_formel2} has been less investigated until recently. Due to the form of equation \eqref{system_lin_formel2}, both the control space and the state space can be chosen of the same nature, such as, for instance, $L^q$ for some $q \in [1, +\infty]$. The notion of relative controllability is the simplest one and it requires steering the system in time $T > 0$ from any given initial condition to any given target $x(T)$ in $\mathbb{R}^d$. That question is now completely understood: after studies handling special cases, Mazanti~\cite{Mazanti2017Relative} gave a necessary and sufficient condition for the relative controllability for the general system \eqref{system_lin_formel2}.

Chitour \emph{et al.}~\cite{Chitour2020Approximate} proposed algebraic conditions for the exact and approximate controllability in $L^2([-\Lambda_N,0],\mathbb{R}^d)$ in uniform time $T$ and gave a bound on the minimal time of controllability in the case where the delays are rationally dependent, situation which is solved by a standard state augmentation technique, reducing the matters at hand to the case of a single delay. They also addressed the first nontrivial case of two rationally independent delays when the system has dimension two and has a single scalar control, i.e., when $N=d=2$ and $m=1$. They provided necessary and sufficient conditions for both approximate and exact controllability in terms of Kalman-type criteria on the parameters of the system and asked whether it is possible to obtain such criteria in a more general situation. However, the techniques of proof used in \cite{Chitour2020Approximate} become intractable when working in higher dimension or with more than two delays.

The approach adopted in this paper to address approximate and exact controllability of System \eqref{system_lin_formel2} is based instead on Yamamoto's realization theory developed in \cite{Yutaka_Yamamoto,yamamoto1981realization,YamamotoRealization,yamamoto1989reachability,Yamamoto_coprimness_measure,YAMAMOTO_Multi_Ring_2011,Yamamoto_Willems}. Yamamoto's theory, which extends older ideas given in \cite{kalman1972realization,kamen1976module}, considers controllability issues for linear systems without a bound on the controllability time by adopting a distributional framework (i.e., inputs and outputs belong to spaces of distributions) in which the system can be characterized by a pair $(P, Q)$ of matrix-valued distributions. In such a framework, controllability issues are given in terms of Bézout's characterizations. More precisely, exact controllability from the origin (in the space of distributions) is equivalent to the existence of two distributions $R, S$ solving Bézout's identity $Q * R + P * S = \delta_0 I_d$ \cite{Yamamoto_Willems}, while $L^2$ approximate controllability  from the origin is equivalent to proving an approximate Bézout identity, i.e., the existence of sequences of distributions $(R_n)_{n \in \mathbb N}$ and $(S_n)_{n \in \mathbb N}$ such that $Q * R_n + P * S_n$ converges to $\delta_0 I_d$ in the distributional sense as $n$ tends to infinity \cite{YamamotoRealization}. In realization theory, the solvability of a Bézout identity (respectively, an approximate Bézout identity) is usually called \emph{left coprimeness} (respectively, \emph{approximate left coprimeness}), cf.\ \cite{polderman1998introduction,sontag}. Note that one of the virtues of relating controllability issues and Bézout's identities is that, if the latter has been solved, then one has a solution for the motion planning problem, i.e., a right-inverse to the endpoint map, cf.\ Proposition~\ref{prop_fond_exact_controllability}.

Our results deal with approximate and exact criteria for controllability in functional state spaces $L^q([-\Lambda_N,0],\mathbb{R}^d)$, for $q \in [1,+\infty)$. First, we prove that the range of the endpoint map from the origin associated with System \eqref{system_lin_formel2} saturates (i.e., does not increase) for $T\geq d\Lambda_N$, enabling one to provide an upper bound on the minimal time of controllability, namely, proving that if $L^q$ approximate (or exact) controllability holds true, then such a controllability must occur in time less than or equal to $d\Lambda_N$ (see Theorem~\ref{lem:3}). This is based on an explicit representation of the endpoint map given in \cite{Chitour2020Approximate}. It should be noticed that, for hyperbolic systems of conservation laws, sharp results on the (exact and null) controllability time have been obtained in the remarkable papers \cite{Coron2019Optimal, Coron2021Null}. We also mention the striking results obtained in \cite{CoronOptimal} on optimal controllability time in the case of hyperbolic systems with analytic time-varying coefficients. We recall, however, that hyperbolic systems of conservation laws correspond to special classes of linear difference delay equations of the form \eqref{system_lin_formel2}.

Furthermore, specifying Yamamoto's Hautus-type approximate controllability criterion to System~\eqref{system_lin_formel2}, the second contribution of this paper consists in providing sufficient and necessary Hautus-type approximate controllability criteria for \eqref{system_lin_formel2} in $L^q([-\Lambda_N,0],\mathbb{R}^d)$ for all $q \in [1,+\infty)$ (see Theorem~\ref{main_result1} and Proposition~\ref{Prop2:main_result0}). In particular, since our criteria are independent of $q$, we deduce that, if $L^q$ approximate controllability holds for some $q \in [1, +\infty)$, then it holds for every such $q$.

The third contribution of this paper is a characterization of the $L^1$ exact controllability
of System~\eqref{system_lin_formel2} in terms of 
a Bézout identity over a Radon measure algebra (see Theorem~\ref{prop_fond_exact_controllability_L1}),
showing in particular that establishing $L^1$ exact controllability
of System~\eqref{system_lin_formel2} is equivalent to solving a corona problem (see \cite{carleson1962interpolations,Fuhrmann_corona} for more details on corona problems). It is worth noticing that the saturation of the range of the endpoint map is a crucial step also to derive such a  characterization. A necessary Hautus-type criterion for the solvability of the Bézout identity is given in Proposition~\ref{prop_nece_bezout}, allowing us to obtain a necessary condition for the exact controllability in $L^1([-\Lambda_N,0],\mathbb{R}^d)$ (see Theorem~\ref{main_result2} and Proposition~\ref{Prop2:main_result}). The necessary condition for the $L^1$ exact controllability can be seen as the closure of the conditions obtained for the $L^q$ approximate controllability.
 
We conjecture that such a necessary condition is also sufficient for the exact controllability in $L^q([-\Lambda_N,0],\mathbb{R}^d)$, for any $q\in[1,+\infty)$. This amounts to solving a Bézout identity over a Radon measure algebra, which turns out to be a challenging open question.
Note that the conjecture holds true for $N=d=2$ and $m=1$ as regards the $L^2$ exact controllability, cf.\ \cite{Chitour2020Approximate}. In addition, we also provide in this paper a result, Theorem~\ref{thm:Ck}, yielding a partial positive answer to the conjecture, stating that our Hautus-type condition implies exact controllability between regular enough functions.

The exact controllability conditions we provide in Proposition~\ref{Prop2:main_result} are reminiscent of other controllability criteria for abstract equations in Banach spaces, such as Condition (24) in \cite{Miller2004}, obtained through semigroup theory. However, such results are strongly related to the skew-adjointness of the involved infinitesimal generator operators and the unitary character of the corresponding semigroups, which is not the case for general difference delay systems of the form \eqref{system_lin_formel2}. We also remark that our approximate controllability criterion is similar to that given in \cite{hale2002stabilization} for the strong stabilization of System \eqref{system_lin_formel2}, i.e., for the existence of matrices $K_1, \dotsc, K_N$ so that the feedback $u(t) = \sum_{j=1}^N K_j x(t - \Lambda_j)$ stabilizes System \eqref{system_lin_formel2} for all choices of delays $\Lambda_1, \dotsc, \Lambda_N$. Namely, combining our result with \cite[Theorem~3.1]{hale2002stabilization}, approximate controllability of \eqref{system_lin_formel2} for every choice of delays $\Lambda_1, \dotsc, \Lambda_N$ implies strong stabilization of System \eqref{system_lin_formel2}.

We conclude the paper by illustrating the applicability of our results. We start by recovering those of \cite{Chitour2020Approximate}, at least for what concerns the characterization of approximate controllability. We then highlight the generality of our criteria by discussing in more details the case of systems with two delays and a single input in dimension three. In that case, we obtain a frequency-free approximate controllability characterization in the spirit of that of \cite{Chitour2020Approximate}.

The sequel of the paper is organized as follows. Section~\ref{sec:notation} introduces the notations used in this article, while our main results are stated in Section~\ref{sec:results}. In Section~\ref{sec:6} we establish the saturation of the range of the endpoint map for times larger than $d \Lambda_N$. Using Yamamoto's realization theory, our two main Hautus--Yamamoto criteria are proved in Section~\ref{section_Hautus_criteria}. Finally, we apply these criteria in Section~\ref{sec:applications} to deduce controllability criteria of Kalman type, recovering previous results in the literature and extending them to more general cases.

\section{Notation}
\label{sec:notation}
In this paper, we denote by $\mathbb{N}$ and $\mathbb{N}^*$ the sets of nonnegative and 
positive integers, respectively. The set $\{1,\dots,N \}$ is represented by $\llbracket 1,N\rrbracket$ for any 
$N \in \mathbb{N}^*$. We use $\mathbb{Z}$, $\mathbb{R}$, $\mathbb{C}$, $\mathbb{R}_+$, 
and $\mathbb{R}_-$ to denote the sets of relative integers, real numbers, complex numbers, 
nonnegative, and nonpositive reals respectively. For $p \in \mathbb{C}$, $\Re(p)$ and 
$\Im(p)$ represent the real and imaginary parts of $p$.
 For $N \in \mathbb{N}^*$ and $n=(n_1,\dots,n_N) \in \mathbb{N}^N$,  the length of the 
 $N$-tuple $n$ is denoted by $|n|$ and is equal to $n_1+ \cdots + n_N$.
For $\Lambda=(\Lambda_1,\dots,\Lambda_N) \in \mathbb{R}^N$, we write $\Lambda \cdot n:=n_1 \Lambda_1+ \cdots+ n_N \Lambda_N$. Given two positive integers $i$ and $j$, $\mathcal{M}_{i,j}(\mathbb{K})$ is the set of $i \times 
j$ matrices with coefficients in $\mathbb{K}=$ $\mathbb{R}$ or $\mathbb{C}$. For $A \in 
\mathcal{M}_{i,j}(\mathbb{K})$, we note $A^\ast$ its conjugate transpose matrix. We use 
$\|\cdot\|$ to denote a norm for every finite-dimensional space (over $\mathbb{K}$) and 
$\vertiii{\,\cdot\,}$ the induced norm for linear maps.

Otherwise stated, elements $x\in \mathbb{K}^i$ are considered as column vectors. The identity matrix in $\mathcal{M}_{i,i}(\mathbb{K})$ is denoted by $I_i$. For $M \in \mathcal{M}_{i,j}(\mathbb{K})$, $\rank M$ denotes the rank of $M$. Given a positive integer $k$, $A \in \mathcal{M}_{i,j}(\mathbb{K})$, and $B\in  \mathcal{M}_{i,k}(\mathbb{K})$, the bracket $\left[A,B\right]$ denotes the juxtaposition of the two matrices, which hence belongs to $\mathcal{M}_{i,j+k}(\mathbb{K})$.

Let $k \in \mathbb{N}^*$ and $q \in [1,+\infty)$. Given an interval $I$ of $\mathbb{R}$, $L^q(I,\mathbb{R}^k)$ represents the space of $q$-integrable functions on the interval $I$ with values in $\mathbb{R}^k$ endowed of the $L^q$-norm on $I$ denoted $\| \cdot \|_{I,\,q}$.
The space of $q$-integrable functions on compact subsets of $\mathbb{R}$ (respectively, $\mathbb{R}_+$) with values in $\mathbb{R}^k$ is denoted $L^q_{\rm loc}\left(\mathbb{R},\mathbb{R}^k\right)$ (respectively,  $L^q_{\rm loc}\left(\mathbb{R}_+,\mathbb{R}^k\right)$). The semi-norms
\[
\|\phi \|_{[0,a],q}:=\left(\int_0^a \vertii{\phi(t)}^q dt  \right)^{1/q},\qquad \phi \in L^q_{\rm loc}\left(\mathbb{R}_+,\mathbb{R}^k\right), \ a\ge 0,
\]
induce a topology on $L^q_{\rm loc}\left(\mathbb{R}_+,\mathbb{R}^k\right)$, which is then a Fr\'echet space.
For a linear operator $f$ we denote $\Ran f$ its range and $\overline{\Ran f}$ the closure of the range. More generally, if $F$ is a matrix-valued holomorphic function, we use $F(\mathbb{C})$ and $\overline{F(\mathbb{C})}$ to denote its image and the the closure of its image, respectively.

We next introduce the distributional framework needed in the paper. A detailed presentation 
with precise definitions can be found, e.g., in 
\cite{schwartz1966theorie,YamamotoRealization,yamamoto1989reachability}. We use   
$\mathcal{D}(\mathbb{R})$ to denote the space of $C^{\infty}$ functions defined on $\mathbb{R}$ with 
compact support, endowed with its canonical LF topology. We also use $\mathcal{D}'(\mathbb{R})$ to denote the space of continuous linear forms acting on $\mathcal{D}(\mathbb{R})$, i.e., the space of all distributions 
on $\mathbb{R}$, endowed with the strong dual topology of uniform convergence on bounded subsets of $\mathcal{D}(\mathbb{R})$. For $\alpha \in\mathcal{D}'(\mathbb{R})$ and $\psi \in \mathcal{D}(\mathbb{R})$, $\scalprod{\alpha,\psi}$ denotes the duality product. The support of a distribution $\alpha \in\mathcal{D}'(\mathbb{R}) $, denoted $\supp( \alpha)$, is the complement of the largest open set on which $\alpha$ is zero. The order of a distribution $\alpha \in\mathcal{D}'(\mathbb{R})$ is the smallest integer $p$ such that, for every compact set $K \subset \mathbb R$, there exists $C_K > 0$ such that
\begin{equation*}
\abs{\scalprod{\alpha,\psi}} \le C_K\, \underset{x \in K}{\sup}\, \abs{\psi^{(p)}(x)},
\end{equation*}
for all $\psi \in \mathcal D(\mathbb R)$ with compact support in $K$, where $\psi^{(p)}$ denotes the $p$-th derivative of $\psi$. We note $\delta_x \in \mathcal{D}'(\mathbb{R})$ the Dirac distribution at $x \in \mathbb{R}$. Notice that $\delta_x$ has order zero.
 
To deal with our controllability issues, we introduce the following subspaces of $\mathcal{D}'(\mathbb{R})$. We use $\mathcal{E}'(\mathbb{R}_{-})$ and $\mathcal{E}'(\mathbb{R}_{+})$ to denote  the spaces of distributions having compact support in $\mathbb{R}_{-}$ and $\mathbb{R}_{+}$, respectively. We denote by $M(\mathbb{R}_{-})$ and $M(\mathbb{R}_{+})$  the subspaces of $\mathcal{E}'(\mathbb{R}_{-})$ and $\mathcal{E}'(\mathbb{R}_{+})$ consisting of distributions of order zero. The spaces $M(\mathbb{R}_{-})$ and $M(\mathbb{R}_{+})$ can  be also characterized as the sets of Radon measures with compact support contained in $\mathbb{R}_{-}$ and $\mathbb{R}_{+}$, thanks to the Riesz representation theorem  (see, e.g., \cite[Theorem~6.19]{Rudin1987Real}). 
Let $\mathcal{D}'_{+}(\mathbb{R})$ be the space of distributions having support bounded on the left, which becomes an algebra when endowed with the convolution product~$*$. We also consider $M_+(\mathbb{R})$ as the space of Radon measures with support bounded on the left, or equivalently the subspace of $\mathcal{D}'_{+}(\mathbb{R})$ of distributions of order zero. With a slight abuse of notation, we also write $\mathcal D'(\mathbb R)$, $\mathcal E'(\mathbb R_-)$, $M(\mathbb R_-)$, $\mathcal D^\prime_+(\mathbb R)$, and $M_+(\mathbb R)$ to refer to sets of matrices whose entries belong to those respective spaces.

Given a Radon measure $\mu \in M_+(\mathbb{R})$, we use $\widehat{\mu}(p)$ to denote the two-sided Laplace transform of $\mu$ at frequency $p\in \mathbb{C}$, that is,
\begin{equation}
\label{laplace_transform_radon_measure}
\widehat{\mu}(p)=\int_{-\infty}^{+\infty} d\mu(t)e^{-pt},
\end{equation}
provided that the integral exists. In particular, the Laplace transform of a Dirac distribution $\delta_x$, $x \in \mathbb{R}$, is the holomorphic map
\begin{equation}
\label{laplace_transform_dirac}
\widehat{\delta}_x:p\mapsto e^{-p x},\qquad p\in \mathbb{C}.
\end{equation}
In the more general setting of a distribution $\alpha \in\mathcal{D}'(\mathbb{R})$, the Laplace transform $
\widehat{\alpha}(p)$ of $\alpha$ at $p\in\mathbb{C}$ is defined by 
$\scalprod{\alpha,e^{- \cdot p}}$, when this quantity makes sense, i.e., when the linear form $\alpha$ on $\mathcal D(\mathbb R)$ can be extended by continuity (in the standard topology of $C^\infty(\mathbb R)$) to the function $t \mapsto e^{- t p}$.

We now define the truncation to positive times of a distribution in $\mathcal D'_+(\mathbb R)$, following the presentation of \cite{YamamotoRealization,yamamoto1989reachability}. Let $\mathcal{D}(\mathbb{R}_+) \subset \mathcal D(\mathbb R)$ be the space of infinitely differentiable functions on $\mathbb{R}$ with compact support contained in $\mathbb{R}_+$ and $\mathcal{D}'(\mathbb{R}_+)$ be its topological dual space. For $\alpha \in \mathcal{D}'_{+}(\mathbb{R})$, we define its truncation $\pi\alpha \in \mathcal D'(\mathbb R_+)$ by
\begin{equation}
\label{eq:pi}
\scalprod{\pi \alpha, \psi} = \scalprod{\alpha, \psi},\qquad \psi \in \mathcal{D}(\mathbb{R}_+).
\end{equation}
We have that $\pi \alpha$ is a well-defined element of $\mathcal{D}'(\mathbb{R}_+)$. Furthermore, $\pi$ is continuous with respect to the strong dual topology of $\mathcal{D}'_{+}(\mathbb{R})$ and $\mathcal{D}'(\mathbb{R}_+)$. Note that $\pi$ truncates distributions to positive times only and, in particular, $\pi\delta_0 = 0$. For further properties concerning the operator $\pi$, see Lemma~\ref{lem:sumlemmasYamamoto}.

\section{Description of the problem and statement of the controllability criteria}
\label{sec:results}

We start by defining solutions of System~\eqref{system_lin_formel2} considered in this paper. 
The following proposition can be easily obtained by a direct step-by-step construction of the 
solution, as given in 
 \cite[Proposition~3.2]{Chitour2016Stability} (cf.~also \cite[Remark 2.3]{Mazanti2017Relative}). Unless otherwise stated, $q$ denotes any real number belonging to the interval $[1,+\infty)$.

\begin{proposition}
\label{Prop1:main_result}
Let $T > 0$, $u \in L^q([0,T],\mathbb{R}^m)$, and $x_0 \in L^q([-\Lambda_N,0],\mathbb{R}^d)$. There exists a unique solution $x \in L^q([-\Lambda_N,T],\mathbb{R}^d)$ such that $x(\theta)=x_0(\theta)$ for all $\theta \in [-\Lambda_N,0]$ and $x(\cdot)$ satisfies Equation~\eqref{system_lin_formel2} for all $t \in [0,T]$. 
\end{proposition}

From now on, given  $T > 0$, $u \in L^q([0,T],\mathbb{R}^m)$, and $x_0 \in L^q([-\Lambda_N,0],\mathbb{R}^d)$, we write $x \in L^q([-\Lambda_N,T],\mathbb{R}^d)$ to denote the solution given by Proposition~\ref{Prop1:main_result}. For all $t \in [0,T]$, we denote by $x_t \in L^q([-\Lambda_N,0],\mathbb{R}^d)$ the function defined by $x_t(\theta):=x(t+\theta)$, $\theta \in [-\Lambda_N,0]$. Using such a notation, we introduce in the next definition the standard notions of approximate and exact controllability in finite time $T>0$ that we are interested to study in this paper (see, for instance, \cite[Chapter~2]{coron}).

\begin{definition}\label{def:Lq-cont} System~\eqref{system_lin_formel2} is said to be:
\begin{enumerate}[1)]
\item \emph{$L^q$ approximately controllable in time $T>0$} if for every $x_0,\phi \in L^q([-\Lambda_N,0],\mathbb{R}^d)$ and $\epsilon>0$, there exists $u \in L^q([0,T],\mathbb{R}^m)$ such that 
\[
\|x_T- \phi\|_{[-\Lambda_N,0],q} < \epsilon;
\]
\item \emph{$L^q$ exactly controllable in time $T>0$} if for every $x_0,\phi \in L^q([-\Lambda_N,0],\mathbb{R}^d)$, there exists $u \in L^q([0,T],\mathbb{R}^m)$ such that 
\[
x_T= \phi.
\]
\end{enumerate}
\end{definition}

\begin{remark}
If System~\eqref{system_lin_formel2} is $L^q$ approximately (respectively, exactly) controllable in time $T>0$ then it is $L^q$ approximately (respectively, exactly) controllable in any time $T'\geq T$.
\end{remark}

We present below weaker notions of controllability, corresponding to those used by Yamamoto \cite{yamamoto1989reachability}, in which the controllability time is not fixed in advance, but may depend on the target to be reached, and the initial state is always assumed to be the origin.

\begin{definition}\label{def:contr-from-origin} System~\eqref{system_lin_formel2} is said to be:
\begin{enumerate}[1)]
\item \emph{$L^q$ approximately controllable (from the origin)} if for $x_0\equiv 0$, every $\phi \in L^q([-\Lambda_N,0],\mathbb{R}^d)$, and $\epsilon>0$, there exist $T_{\epsilon,\phi}>0$ and $u \in L^q([0,T_{\epsilon,\phi}],\mathbb{R}^m)$ such that
\[
\|x_{T_{\epsilon,\phi}}- \phi\|_{[-\Lambda_N,0],q} < \epsilon;
\]
\item \emph{$L^q$ exactly controllable (from the origin)} if for $x_0 \equiv 0$ and every  $\phi \in L^q([-\Lambda_N,0],\mathbb{R}^d)$, there exist $T_{\phi}>0$ and $u \in L^q([0,T_{\phi}],\mathbb{R}^m)$ such that
\[
x_{T_{\phi}}= \phi.
\]
\end{enumerate}
\end{definition}

Note that obvious implications hold true between the different notions provided in Definitions~\ref{def:Lq-cont} and \ref{def:contr-from-origin}, i.e., approximate (respectively, exact) controllability in time $T$ implies approximate (respectively, exact) controllability from the origin.
One of the results of the present paper is that one actually has equivalence between such notions, as stated in the next theorem.

\begin{theorem}
\label{lem:3}
Let  $q\in [1,+\infty)$. System~\eqref{system_lin_formel2} is $L^q$ approximately (respectively, exactly) controllable from the origin if and only if it is $L^q$ approximately (respectively, exactly) controllable in time $T = d \Lambda_N$.
\end{theorem}

The proof of Theorem~\ref{lem:3} is provided in Section~\ref{sec:6}.

\subsection{Hautus--Yamamoto criteria}

Before stating our two main theorems, which give necessary and sufficient (respectively, necessary) Hautus--Yamamoto criteria to ensure the approximate (respectively, exact) controllability in time $d \Lambda_N$ of System~\eqref{system_lin_formel2}, let us introduce the matrix-valued holomorphic map 
\begin{equation}\label{eq:defH}
H(p):=I_d- \sum_{j=1}^N e^{- p \Lambda_j} A_j, \qquad p \in \mathbb{C}.
\end{equation}

The holomorphic function $H$ defined above arises when one considers solutions of \eqref{system_lin_formel2} with $u \in L^q(\mathbb R_+, \mathbb R^m)$ and an initial condition $x(\theta) = 0$ for $\theta \in [-\Lambda_N, 0]$. Indeed, extending $x$ and $u$ by zero for all negative times, \eqref{system_lin_formel2} is satisfied for every $t \in \mathbb R$ and then, taking the two-sided Laplace transform, we obtain that
\begin{equation}
\label{eq:relation_simple_H1}
\widehat{x}(p)= \mathcal{H}(p) \widehat{u}(p),
\end{equation}
for $p$ in some  right-half plane of $\mathbb{C}$, where $ \mathcal{H}(p)=H(p)^{-1} B$ is the transfer function of System~\eqref{system_lin_formel2}, i.e., the matrix describing the linear relation between the frequencies of the input and of the output. We notice that the Laplace transform of $x$ exists for all frequencies in a suitable  right-half plane of $\mathbb{C}$ because the solutions of \eqref{system_lin_formel2} grow at most exponentially. 

Our main results are given next.

\begin{theorem}
\label{main_result1}
{Let $q\in [1,+\infty)$.} System~\eqref{system_lin_formel2} is $L^q$ approximately controllable in time $d \Lambda_N$ if and only if 
the two following conditions hold true:
\begin{enumerate}[i)]
\item \label{assumption1bis}$\rank\left[H(p),B\right]=d$ for every $p \in \mathbb C$,
\item \label{assumption2bis} $\rank[A_N,B]=d$.
\end{enumerate}
\end{theorem}

\begin{theorem}
\label{main_result2} If System~\eqref{system_lin_formel2} is $L^1$ exactly controllable in time $d \Lambda_N$ then the two following conditions hold true:
\begin{enumerate}[i)]
\item \label{assumption1-L1}$\rank\left[M,B\right]=d$ for every  $M\in \overline{H(\mathbb{C})}$,
\item \label{assumption2-L1} $\rank[A_N,B]=d$.
\end{enumerate}
\end{theorem}

Theorems~\ref{main_result1} and \ref{main_result2} are proved in Sections~\ref{subsec:Approximate_contr_HY} and \ref{subsec:Exact_contr_HY}, respectively.

Theorem~\ref{main_result1} is a complete characterization of $L^q$ approximate controllability, providing a necessary and sufficient condition which can be seen as the counterpart for difference equations of Hautus controllability criterion. In particular, the property of $L^q$ approximate controllability in time $d \Lambda_N$ does not depend on $q\in [1, +\infty)$, in the sense that if it holds for some $q\in [1, +\infty)$ then it holds for all of them.
On the contrary, Theorem~\ref{main_result2} provides only a necessary condition for $L^q$ exact controllability in the case $q = 1$. We expect in fact that the above condition is also sufficient for any $q$ and, in that direction, we propose the following conjecture.

\begin{conjecture}
\label{conj:HY-exact}
Let $q\in [1,+\infty)$. System~\eqref{system_lin_formel2} is $L^q$ exactly controllable in time $d \Lambda_N$ if and only if the two following conditions hold true:
\begin{enumerate}[i)]
\item\label{assumption1-conject} $\rank\left[M,B\right]=d$ for every  $M\in \overline{H(\mathbb{C})}$,
\item\label{assumption2-conject} $\rank[A_N,B]=d$.
\end{enumerate}
\end{conjecture}

This is motivated firstly by the fact that this conjecture actually holds true for the $L^2$ exact controllability in the case $N = d = 2$ and $m=1$, as it follows from the results of \cite{Chitour2020Approximate} (see  Section~\ref{sec:applications} for details). In addition, we provide below a result giving ground to the above conjecture. For that purpose, we introduce the following compatibility condition for the existence of regular enough solutions of \eqref{system_lin_formel2}.

\begin{definition}\label{def:compatibility}
Let $k \in \mathbb N$ and $x_0 \in C^k([-\Lambda_N, 0], \mathbb R^d)$. We say that $x_0$ is \emph{$C^k$-admissible} for System \eqref{system_lin_formel2} if, for every integer $\ell$ with $0 \leq \ell \leq k$, we have
\[
x_0^{(\ell)}(0) = \sum_{j=1}^N A_j x_0^{(\ell)}(-\Lambda_j).
\]
\end{definition}

\begin{theorem}
\label{thm:Ck}
Assume that Conditions \ref{assumption1-conject} and \ref{assumption2-conject} of Conjecture~\ref{conj:HY-exact} hold true. Then there exists a nonnegative integer $k$ such that, for every $q \in [1, +\infty)$ and $x_0, \phi \in C^k([-\Lambda_N, 0], \mathbb R^d)$ $C^k$-admissible for System \eqref{system_lin_formel2}, there exists $u \in L^q([-\Lambda_N, 0], \mathbb R^m)$ such that $x_{d \Lambda_N} = \phi$ almost everywhere in $[-\Lambda_N, 0]$.
\end{theorem}

The proof of Theorem~\ref{thm:Ck} is provided in Section~\ref{subsec:Ck}.

\subsection{Further controllability characterizations}

The approximate and exact controllability of System~\eqref{system_lin_formel2} can be described in several equivalent ways using Propositions~\ref{Prop2:main_result0} and~\ref{Prop2:main_result}, respectively.

\begin{proposition}
\label{Prop2:main_result0}
The following three statements are equivalent:
\begin{enumerate}[(a)]
\item\label{app-a} $\rank\left[H(p),B\right]=d$ for every  $p \in \mathbb{C}$.

\item\label{app-b} For every $p \in \mathbb{C}$, one has
\[
 \inf \left\{ \vertii{g^TH(p)}+\vertii{g^TB} \suchthat  g\in \mathbb{C}^d,\ \|g\|=1 \right\} >0.
\]
\item\label{app-b'} 
For every $p \in \mathbb{C}$, one has
\[
\det \left( H(p) H^\ast(p)+BB^\ast \right) >0.
\]
\end{enumerate}

\end{proposition}
Proposition~\ref{Prop2:main_result} below is the counterpart of Proposition~\ref{Prop2:main_result0} for the case of exact controllability.

\begin{proposition}
\label{Prop2:main_result}
Under the assumption that $\rank\left[A_N,B\right]=d$, 
the following three statements are equivalent:
\begin{enumerate}[(a)]
\item\label{exa-a} $\rank\left[M,B\right]=d$ for every $M\in \overline{H(\mathbb{C})}$.

\item\label{exa-b} There exists $\alpha>0$ such that, for every $p\in \mathbb{C}$,  
\[
 \inf \left\{ \vertii{g^TH(p)}+\vertii{g^TB} \suchthat \mbox{$g\in \mathbb{C}^d$, $\|g^T \|=1$} \right\} \ge \alpha. 
\]
\item\label{exa-b'}
There exists $\alpha>0$ such that, for every $p \in \mathbb{C}$,
\[
\det \left( H(p) H^\ast(p)+BB^\ast  \right) \ge \alpha .
\]
\end{enumerate}

\end{proposition}

The proofs of Propositions~\ref{Prop2:main_result0} and \ref{Prop2:main_result} are given in Section~\ref{subsec:useful_properties}. 

As an immediate consequence of Theorems~\ref{main_result1} and \ref{main_result2} and Propositions~\ref{Prop2:main_result0} and \ref{Prop2:main_result}, we have the following corollary. 

\begin{corollary}
Under the assumption that $\rank\left[A_N,B\right]=d$,  we have:
\begin{enumerate}[i)]
    \item 
Given any $q \in [1,+\infty)$, System~\eqref{system_lin_formel2} is $L^q$ approximately controllable in time $d\Lambda_N$ 
if and only if one of the items~\ref{app-a}--\ref{app-b'} of Proposition~\ref{Prop2:main_result0}
holds true.
\item If System~\eqref{system_lin_formel2} is $L^1$ exactly controllable in time $d\Lambda_N$ 
then all items~\ref{app-a}--\ref{app-b'} of Proposition~\ref{Prop2:main_result} hold true.
\end{enumerate}
\end{corollary}

Thanks to Item~\ref{exa-b'} in  Propositions \ref{Prop2:main_result0} and \ref{Prop2:main_result}, Item~\ref{assumption1bis} of Theorems~\ref{main_result1} and \ref{main_result2} can be reformulated in an equivalent way in terms of 
the non-vanishing properties of the function $p\mapsto \det \left( H(p) H^\ast(p)+BB^\ast  \right)$. Our next result shows that this function can only vanish in a bounded vertical strip of the complex plane.

\begin{proposition}
\label{prop_zero_Q0}
There exist $\beta_1,\beta_2,\rho\in\mathbb{R}$ such that
\begin{equation}
\label{eq:ineg-H-H*}
\det \left(H(p) H^\ast(p) + B B^\ast\right)\ge \abs{\det \left(H(p)\right)}^2 \ge \rho > 0
\end{equation}
for all $p \in \mathbb C$ such that $\Re(p) < \beta_1$ or $\Re(p) > \beta_2$.
\end{proposition}

\begin{proof}
The left inequality in \eqref{eq:ineg-H-H*} is trivial.
As for the second one, notice first that one can assume, with no loss of generality, that $H(\cdot)\not\equiv I_d$. Moreover, for $\Re(p)$ large enough,
$H(p)$ is arbitrarily close to the identity. This implies that there exists $\beta_2>0$ such that $\abs{\det \left(H(p)\right)} \geq 1/2$ for $\Re(p)> \beta_2$.

On the other hand, the holomorphic function $h$ defined by $h(p) = \det \left(H(p)\right)$
can be written as
\begin{equation*}
h(p)=\tilde{h}(p)+h_n e^{-p(\Lambda \cdot n)},\qquad p\in \mathbb{C},
\end{equation*}
where 
$h_n$ is a nonzero real number,
$n$ is in $\mathbb{N}^N$ and satisfies $0 \le |n|\le N$,
and
$\tilde{h}(\cdot)$ is a holomorphic function such that $\displaystyle\lim_{\Re(p) \to -\infty} \tilde{h}(p)e^{p(\Lambda \cdot n)}=0$. 
We then deduce that 
$h_n e^{-p(\Lambda \cdot n)}$ is the dominant term of $h(p)$ as $\Re(p) \to -\infty$,
and it goes to $\infty$ in modulus. 
This implies that for every $\rho>0$ there exists $\beta_1\in \mathbb{R}$ such that $|h(p)|^2 \ge \rho$ for $\Re(p)<\beta_1$.
\end{proof}

\section{Properties of the endpoint map and consequences for controllability}
\label{sec:6}

The goal of this section is threefold. Firstly, we give a variation-of-constants formula for System~\eqref{system_lin_formel2} allowing us to write the solution operator as the sum of a flow operator and an endpoint map. As a consequence, we can give characterizations of the controllability of System~\eqref{system_lin_formel2} in terms of the range of the endpoint map.
Secondly, through a Cayley--Hamilton theorem for multivariate polynomials, we show that the range of the endpoint map is constant from the time $d\Lambda_N$ and finally we deduce that the approximate (exact respectively) controllability of System~\eqref{system_lin_formel2} is equivalent to the approximate controllability (exact respectively) in finite time $d\Lambda_N$, thus providing a proof for Theorem~\ref{lem:3}.

\subsection{Definitions and preliminary remarks}
\label{sec:4}

We recall in Proposition~\ref{prop_var_const_formula} the explicit representation formula for solutions
of System~\eqref{system_lin_formel2}, often called variation-of-constants formula (and sometimes flow formula), which can be found in \cite{Chitour2020Approximate}. Before the statement of the proposition, let us give the following definitions.

\begin{definition}
\label{def_Xi_n}
The family of matrices $\Xi_n \in \mathcal{M}_{d,d}(\mathbb{R})$, $n \in \mathbb{Z}^N$, is defined inductively as
\begin{equation*}
 \Xi_n= \begin{cases}
     0 & \mbox{if $n \in \mathbb{Z}^N \backslash \mathbb{N}^N$},\\
      I_d & \mbox{if  $n=0$},\\
 \sum_{k=1}^N A_k \Xi_{n-e_k} & \text{if $n \in \mathbb{N}^N$ and $|n|>0$},
    \end{cases}   
\end{equation*}
where $e_k$ denotes the $k$-th canonical vector of $\mathbb Z^N$ for $k \in \llbracket 1,N\rrbracket$, i.e., all the coordinates of $e_k$ are zero except for the $k$-th one, which is equal to one.
\end{definition}
\begin{definition}
\label{def_Upsilon_E}
 For $T \in [0,+\infty)$, we introduce the following two  operators:

\begin{enumerate}[1)]
\item The \emph{flow operator $\Upsilon_{q}(T)\colon L^q([-\Lambda_N,0],\mathbb{R}^d) \to L^q([-\Lambda_N,0],\mathbb{R}^d)$ of System~\eqref{system_lin_formel2}} defined by
\[
\left(\Upsilon_{q}(T)x_0 \right)(s)=\sum_{\substack{(n,j) \in \mathbb{N}^N \times \llbracket1,N\rrbracket \\ -\Lambda_j\le  T+s-\Lambda \cdot n <0}}  \Xi_{n-e_j} A_j x_0(T+s-\Lambda \cdot n),
\]
for $x_0 \in L^q([-\Lambda_N,0],\mathbb{R}^d)$ and $s \in [-\Lambda_N,0]$. The operator $\Upsilon_{q}(T)$ represents the solution operator of System~\eqref{system_lin_formel2} with $B = 0$ and initial state $x_0$.
\item The \emph{endpoint operator $E_{q}(T)\colon L^q([0,T],\mathbb{R}^m) \to L^q([-\Lambda_N,0],\mathbb{R}^d)$ of System~\eqref{system_lin_formel2}} defined by 
\[
\left(E_{q}(T)u \right)(t)= \sum_{\substack{n \in \mathbb{N}^N \\ \Lambda \cdot n \le T+t}}  \Xi_n B u(T+t-\Lambda \cdot n),
\]
for $u \in L^q([0,T],\mathbb{R}^m)$ and $t\in [-\Lambda_N,0]$.
\end{enumerate}
\end{definition} 

\begin{proposition}[Variation-of-constants formula]
\label{prop_var_const_formula}
For $T \in [0,+\infty)$, $u \in L^q([0,T],\mathbb{R}^m)$, $x_0 \in L^q([-\Lambda_N,0],\mathbb{R}^d)$, and $t \in [0,T]$, we have
\begin{equation}
\label{eq:variation_constant}
x_t=\Upsilon_{q}(t)x_0+E_{q}(t)u.
\end{equation}
\end{proposition}

\begin{remark}
Notice that Proposition~\ref{prop_var_const_formula} has been stated for $q=2$ in \cite{Chitour2020Approximate} only but it holds true as well for $q \in [1,+\infty)$.
\end{remark}

In Proposition~\ref{prop:first_charact_control} below we use the variation-of-constants formula to express approximate and exact controllability in terms of the image of the operator $E_{q}(T)$, $T>0$. 
\begin{proposition}
\label{prop:first_charact_control}
Let $q \in [1, +\infty)$ and $T > 0$.
\begin{enumerate}[i)]
\item System~\eqref{system_lin_formel2} is $L^q$ approximately controllable in time $T>0$ if and only if $\Ran E_{q}(T)$ is dense in $L^q([-\Lambda_N, 0], \mathbb R^d)$.
\item System~\eqref{system_lin_formel2} is $L^q$ approximately controllable from the origin if and only if
\begin{equation}
\label{eq_app_con1}
\overline{  \bigcup_{T \ge 0} \Ran E_{q}(T)}=L^q([-\Lambda_N,0],\mathbb{R}^d).
\end{equation}
\item System~\eqref{system_lin_formel2} is $L^q$ exactly controllable in time $T>0$ if and only if $\Ran E_{q}(T) = L^q([-\Lambda_N, 0], \mathbb R^d)$.
\item System~\eqref{system_lin_formel2} is $L^q$ exactly controllable from the origin if and only if
\begin{equation}
\label{eq_app_con2}
\bigcup_{T \ge 0} \Ran E_{q}(T)=L^q([-\Lambda_N,0],\mathbb{R}^d).
\end{equation}
\end{enumerate}
\end{proposition}

\subsection{Saturation of the range of the endpoint map in time \texorpdfstring{$d\Lambda_N$}{d Lambda N}}

Cayley--Hamilton theorem is instrumental to study controllability properties of System~\eqref{system_lin_formel2} containing one delay (see \cite[Remark 3.5]{Chitour2020Approximate}). It turns out that the following generalization of Cayley--Hamilton theorem in the case of multivariate polynomials plays a similar role in our subsequent arguments.
\begin{lemma}
\label{lem:1}
Let $\Xi_n$, $n\in \mathbb{Z}^N$, be the matrices introduced in Definition~\ref{def_Xi_n}. 
There exist real coefficients $\alpha_k$, for 
$k\in \mathbb{N}^N$ with $0<|k| \le d$, such that, for every $n\in \mathbb{N}^N$ with $|n|\ge d$,
\begin{equation}
\label{eq:lem1:1}
\Xi_n=-\sum_{\substack{k \in \mathbb{N}^N \\ 0<|k| \le d}}\alpha_k \Xi_{n-k}.
\end{equation}
\end{lemma}

\begin{proof}
For $t=(t_1,\dots,t_N) \in \mathbb{R}^N$, set
$$
A(t)=t_1A_1+t_2A_2+ \cdots+t_N A_N.
$$
One deduces, by Definition~\ref{def_Xi_n} and an immediate  induction argument that, for 
every $j\in \mathbb{N}$ and $t \in \mathbb{R}^N$, it holds
\begin{equation}
\label{eq:lem1:1.1}
A(t)^j=\sum_{\substack{n \in \mathbb{N}^N \\ |n|=j}} \Xi_n t^n,
\end{equation}
where $t^n:=t_1^{n_1}  t_2^{n_2} \dotsm t_N^{n_N}$. 
Using Neumann series, we deduce from Equation~\eqref{eq:lem1:1.1} 
that, for $t$  small enough,
\begin{align}
\big(I_d-A(t)\big)^{-1}&=\sum_{j\in \mathbb{N}} A(t)^j
=\sum_{n\in \mathbb{N}^N} \Xi_{n}t^n.\label{eq:lem1:2}
\end{align}
Notice that $P(t)=\det\big(I_d-A(t)\big)$ is a multivariate polynomial of degree $d$, that is,
\begin{equation}
\label{eq:lem1:3}
P(t)=\sum_{\substack{k \in \mathbb{N}^N \\ 0 \le |k| \le d}} \alpha_k t^k,
\end{equation}
for some real numbers 
$\alpha_k$ defined for $k \in \mathbb{N}^N$ such that $|k| \le d$  and $\alpha_0=1$. Let $\operatorname{Adj}\big(I_d-A(t)\big)$ be the adjugate matrix of $I_d-A(t)$.
We have, on the one hand, that there exist $M_k \in \mathcal{M}_{d, d}(\mathbb{R})$ for $k \in \mathbb{N}^N$ and $0 \le|k|\le d-1$ such that
\begin{equation}
\label{eq:lem1:4}
\operatorname{Adj}\big(I_d-A(t)\big)=\sum_{\substack{k \in \mathbb{N}^N \\ 0 \le |k| \le d-1}}  M_{k}t^k
\end{equation}
and, on the other hand, that Equations~\eqref{eq:lem1:2}--\eqref{eq:lem1:3} lead to
\begin{align}
\operatorname{Adj}\big(I_d-A(t)\big)&= P(t)\big(I_d-A(t)\big)^{-1}
=\sum_{\substack{k \in \mathbb{N}^N \\ 0 \le |k| \le d}} \alpha_{k} t^k\sum_{n \in \mathbb{N}^N} \Xi_{n}t^n
= \sum_{n \in \mathbb{N}^N} \sum_{\substack{k \in \mathbb{N}^N \\ 0 \le |k| \le d}} \alpha_k \Xi_{n} t^{n+k}.
\label{eq:lem1:5}
\end{align}
The substitution $l=n+k$ in \eqref{eq:lem1:5} allows us to write
\begin{equation}
\label{eq:lem1:6}
\operatorname{Adj}\big(I_d-A(t)\big)= \sum_{n \in \mathbb{N}^N} \sum_{\substack{l \in \mathbb{N}^N,\,l-n \in \mathbb{N}^N \\ 0 \le \abs{l-n} \le d}} \alpha_{l-n} \Xi_{n} t^{l}.
\end{equation}
We deduce from Equations~\eqref{eq:lem1:4}--\eqref{eq:lem1:6} that, for $l \in \mathbb{N}^N$ and $|l| \ge d$,
\begin{equation}
\label{eq:lem1:7}
\sum_{\substack{n \in \mathbb{N}^N,\,l-n \in \mathbb{N}^N \\ 0 \le |l-n| \le d}} \alpha_{l-n} \Xi_{n} =0.
\end{equation}
Setting $n'=l-n$ in \eqref{eq:lem1:7}, one obtains
\begin{equation}
\label{eq:lem1:8}
\sum_{\substack{n' \in \mathbb{N}^N,\,l-n' \in \mathbb{N}^N \\ 0 \le |n'| \le d}} \alpha_{n'} \Xi_{l-n'} =0.
\end{equation}
Since $\Xi_{l-n'}=0$ for $l-n' \in \mathbb{Z}^N \backslash \mathbb{N}^N$ (see Definition~\ref{def_Xi_n}), we deduce from Equation~\eqref{eq:lem1:8} that
\begin{equation}
\label{eq:lem1:9}
\Xi_{l}=-\sum_{\substack{n' \in \mathbb{N}^N \\ 0 < |n'| \le d}} \alpha_{n'} \Xi_{l-n'},
\end{equation}
hence the conclusion.
\end{proof}

We prove now that the range of the operator $E_{q}(T)$ is constant with respect to $T$ for $T\ge d\Lambda_N$.

\begin{theorem}
\label{lem:RanE_con}
For all $T\in [d \Lambda_N,+\infty)$ and  $q \in [1,+\infty)$, we have
\begin{equation}
\label{eq:RanE_con1}
\Ran E_{q}(T)=\Ran E_{q}(d\Lambda_N).
\end{equation}
\end{theorem}

\begin{proof}
Let $q \in [1,+\infty)$. The proof is divided in two steps. We first prove that Equation~\eqref{eq:RanE_con1} is satisfied for $ T\in [d \Lambda_N,d\Lambda_N+\delta]$ and some $\delta>0$. We then deduce from the flow formula \eqref{eq:variation_constant} that Equation~\eqref{eq:RanE_con1} is actually satisfied for all $T\in [d \Lambda_N,+\infty)$.

Let $T>d \Lambda_N$ and $u \in L^{q}([0,T], \mathbb{R}^m)$.
We define
\begin{equation}
\label{eq:RanE_con2}
u_1(s):=u(s+T-d\Lambda_N),\qquad \mbox{ $s \in [0,d \Lambda_N]$,}
\end{equation}
and
\begin{equation}
\label{eq:RanE_con3}
u_2(s):=  -\sum_{\substack{0<|k| \le d,\, k \in \mathbb{N}^N \\ s<\Lambda\cdot k \le s+T-d \Lambda_N }} \alpha_k u(s-\Lambda \cdot k+T-d \Lambda_N), \qquad \mbox{$s \in [0,d\Lambda_N]$},
\end{equation}
where the coefficients  $\alpha_k$ are defined as in Proposition~\ref{lem:1}. The sum in \eqref{eq:RanE_con3} is understood to be zero when the indices are taken in an empty set. For $t \in [-\Lambda_N,0]$, we have
\begin{align}
\left(E_{q}(T)u \right)(t)&= \sum_{\substack{n \in \mathbb{N}^N \\ \Lambda \cdot n \le {T+t}}}  \Xi_n B u(T+t-\Lambda \cdot n) \nonumber\\
&=\sum_{\substack{n \in \mathbb{N}^N \\ \Lambda \cdot n \le d \Lambda_N+t}}  \Xi_n B u(T+t-\Lambda \cdot n)+\sum_{\substack{n \in \mathbb{N}^N \\ d \Lambda_N+t<\Lambda \cdot n \le T+t}}  \Xi_n B u(T+t-\Lambda \cdot n).
\label{eq:RanE_con4}
\end{align}
Notice that
\begin{equation}
\label{eq:RanE_con5}
\sum_{\substack{n \in \mathbb{N}^N \\ \Lambda \cdot n \le d \Lambda_N+t}}  \Xi_n B u(T+t-\Lambda \cdot n)=E_{q}(d\Lambda_N)u_1(t).
\end{equation}

Since $d \Lambda_N+t<\Lambda \cdot n$ for $n \in \mathbb{N}^N$ implies that $|n|\ge d$, we deduce from 
 Lemma~\ref{lem:1} that
\begin{equation}
\label{eq:RanE_con6}
\sum_{\substack{n \in \mathbb{N}^N \\ d \Lambda_N+t<\Lambda \cdot n \le T+t}}  \Xi_n B u(T+t-\Lambda \cdot n)
=-\sum_{\substack{n \in \mathbb{N}^N \\ d \Lambda_N+t<\Lambda \cdot n \le T+t}}
\sum_{\substack{k \in \mathbb{N}^N \\ 0<|k| \le d }}\alpha_k \Xi_{n-k} B u(T+t-\Lambda \cdot n)
.
\end{equation}

The substitution $n'=n-k$ in Equation~\eqref{eq:RanE_con6} yields
\begin{equation}
\label{eq:RanE_con7}
\sum_{\substack{n \in \mathbb{N}^N \\ d \Lambda_N+t<\Lambda \cdot n \le T+t}}  \Xi_n B u(T+t-\Lambda \cdot n)
=-\sum_{\substack{k \in \mathbb{N}^N \\ 0<|k| \le d }}
\sum_{\substack{n' \in \mathbb{N}^N \\ d \Lambda_N+t<\Lambda \cdot (n'+k) \le T+t}}  \alpha_k \Xi_{n'} B u(T+t-\Lambda \cdot (n'+k))
.
\end{equation}

Let $\delta>0$ be such that for all $T\in [d \Lambda_N,d \Lambda_N+\delta]$, $t \in [-\Lambda_N,0]$, $k \in \mathbb{N}^N$ with $0<|k|\le d$, 
and  $n' \in \mathbb{N}^{\mathbb{N}}$, if $d \Lambda_N+t<\Lambda \cdot (n'+k) \le T+t$ then $\Lambda\cdot n' \le d \Lambda_N+t$.  

Letting $T\in [d \Lambda_N,d \Lambda_N+\delta]$, we can thus rewrite Equation~\eqref{eq:RanE_con7} as
\begin{align}
\MoveEqLeft[6] \sum_{\substack{n \in \mathbb{N}^N \\ d \Lambda_N+t<\Lambda \cdot n \le T+t}}  \Xi_n B u(T+t-\Lambda \cdot n) \nonumber \\
& = -\sum_{\substack{n' \in \mathbb{N}^N \\ \Lambda \cdot n' \le d \Lambda_N+t}}\Xi_{n'}B
\sum_{\substack{k \in \mathbb{N}^N,\, 0<|k| \le d \\ d \Lambda_N+t<\Lambda \cdot (n'+k) \le T+t}} \alpha_k  u(T+t-\Lambda \cdot (n'+k))
\nonumber\\
& = E_{q}(d \Lambda_N)u_2(t).
\label{eq:RanE_con8}
\end{align}
Equations~\eqref{eq:RanE_con4}, \eqref{eq:RanE_con5}, and \eqref{eq:RanE_con8} prove that, for $T \in [d\Lambda_N,d\Lambda_N+\delta]$ and $t \in [-\Lambda_N,0]$,
\begin{equation}
\label{eq:RanE_con9}
E_{q}(T)u(t)=E_{q}(d\Lambda_N)u_1(t)+E_{q}(d \Lambda_N)u_2(t).
\end{equation}
From Equation~\eqref{eq:RanE_con9}, we deduce that, 
\begin{equation}
\label{eq:RanE_con10}
\Ran E_{q}(T)=\Ran E_{q}(d\Lambda_N), \qquad\mbox{ $T \in [d \Lambda_N,d\Lambda_N+\delta]$}.
\end{equation}

Let us now extend Equation~\eqref{eq:RanE_con10} to all  $T \in [d \Lambda_N,+\infty)$. Let $V=\Ran E_{q}(d \Lambda_N)$ and $x \in V$. Fix  $u\in L^q([0,d\Lambda_N],\mathbb{R}^m)$ such that $x=E_{q}(d\Lambda_N)u$. For $t\in [0,\delta]$, define $\tilde{u}\in L^q([0,d\Lambda_N+t],\mathbb{R}^m)$ by setting 
$\tilde{u}|_{[0,d\Lambda_N]}=u$
and ${\tilde{u}}(s)=0$ for $s \in [d \Lambda_N,d\Lambda_N+t]$. From the variation-of-constants formula \eqref{eq:variation_constant}, we have
\begin{equation*}
\Upsilon_{q}(t) x= E_{q}(d\Lambda_N+t)\tilde{u} \in \Ran E_{q}(d\Lambda_N+t).
\end{equation*}
Thanks to Equation~\eqref{eq:RanE_con10} we have proved that
\begin{equation}\label{eq:invariance}
\Upsilon_{q}(t)x \in V,\qquad t\in[0,\delta],\ x\in V.
\end{equation}

Let $y \in \Ran E_{q}(T)$ for $T \in [d\Lambda_N+\delta,d\Lambda_N+2 \delta]$ and $u \in L^q([0,T],\mathbb{R}^m)$ such that $y=E_{q}(T)u$. Define $z=E_{q}(d\Lambda_N+\delta)u|_{[0,d\Lambda_N+\delta]} \in V$. The variation-of-constants formula (Equation~\eqref{eq:variation_constant}) gives
\begin{equation}
\label{eq_lem3:4}
y=\Upsilon_{q}(T-d\Lambda_N-\delta)z+E_{q}(T-d\Lambda_N-\delta)\check{u},
\end{equation}
where $\check{u}(\alpha)=u(\alpha+d\Lambda_N+\delta)$ for $\alpha \in [0,T-d\Lambda_N-\delta]$. We deduce 
from \eqref{eq:RanE_con10} and \eqref{eq:invariance}
that $y \in \Ran E_{q}(d\Lambda_N)$, proving  that 
\begin{equation}
\label{eq_lem3:5}
\Ran E_{q}(d \Lambda_N)=\Ran E_{q}(T), \qquad T \in [d\Lambda_N+\delta,d \Lambda_N+2 \delta].
\end{equation}
The iteration of the same process proves that Equation~\eqref{eq_lem3:5} actually holds for all $T\ge d \Lambda_N$.
\end{proof}

\subsection{Upper bound on the minimal control time}

We next provide a  proof of Theorem~\ref{lem:3}, which is a direct consequence 
of the saturation of the range of the endpoint map from time $d\Lambda_N$.

\begin{proof}[Proof Theorem~\ref{lem:3}]
By Theorem~\ref{lem:RanE_con} and since $T\mapsto \Ran E_{q}(T)$ is monotone nondecreasing for the inclusion, we have that 
\[\bigcup_{T \ge 0} \Ran E_{q}(T)=\Ran E_{q}(d\Lambda_N).\]
The conclusion directly follows from Proposition~\ref{prop:first_charact_control}. 
\end{proof}
One of the major questions concerning the approximate and exact controllability in finite time $T$ is to determine the \textit{minimal time} of controllability.

\begin{definition} We define $T_{\min,\,\mathrm{ap},\,q}$ and $T_{\min,\,\mathrm{ex},\,q}$ the \textit{minimal time} of the approximate and exact controllability  respectively as follows:
\begin{align*}
T_{\min,\,\mathrm{ap},\,q}&:=\underset{T \in \mathbb{R}_+}{\inf} \{ \mbox{System~\eqref{system_lin_formel2} is $L^q$ approximately controllable in time $T$}\},\\
T_{\min,\,\mathrm{ex},\,q}&:=\underset{T \in \mathbb{R}_+}{\inf} \{ \mbox{System~\eqref{system_lin_formel2} is $L^q$ exactly controllable in time $T$} \},
\end{align*}
with the convention that $\inf\emptyset = +\infty$.
\end{definition}
Since an immediate inspection of System~\eqref{system_lin_formel2} shows that it is never approximately or exactly controllable before the time $\Lambda_N$, we can recast Theorem~\ref{lem:3} in terms of minimal time of controllability as follows.
\begin{corollary}
\label{Corollary_min_con}
Both times $T_{\min,\,\mathrm{ap},\,q}$ and $T_{\min,\,\mathrm{ex},\,q}$ belong to the set  $[\Lambda_N,d \Lambda_N] \cup \{+\infty\}$.
\end{corollary}

In the case $N=d=2$ and $m=1$, it is proved in \cite{Chitour2020Approximate} that either $T_{\min,\,\mathrm{ap},\,2}=+\infty$ or $T_{\min,\,\mathrm{ap},\,2}=T_{\min,\,\mathrm{ex},\,2} = 2 \Lambda_{N}$. In the general case however, $d \Lambda_N$ is not always minimal as proved in \cite{Chitour2020Approximate} 
by looking at the case of commensurable delays.  
It was noticed in the same reference that 
$d \Lambda_N$ is minimal
when controllability holds for systems with  a single input and commensurable delays. We conjecture that this result holds true also for the case of non-necessarily commensurable delays. It would be also interesting to investigate the question of equality between $T_{\min,\,\mathrm{ap},\,q}$ and $T_{\min,\,\mathrm{ex},\,q}$ when both of them are finite.

\section{Hautus--Yamamoto criteria for approximate and exact controllability}
\label{section_Hautus_criteria}

In this section, we stick to Yamamoto's notations used in \cite{YamamotoRealization,yamamoto1989reachability}, which have been introduced in Section~\ref{sec:notation}. 

\subsection{Realization theory}
\label{subsec:real_theory}

When the initial condition is fixed at $0$ (which is the case in Definition~\ref{def:contr-from-origin}), System~\eqref{system_lin_formel2} describes a linear relation between the control $u$ on $[0, T]$ and the state $x$ on $[T - \Lambda_N, T]$, described by the operator $E_q(T)$ introduced in Definition~\ref{def_Upsilon_E}. The main idea in realization theory is to represent such a linear relation between an \emph{input} (here the control $u$) and an \emph{output} 
(here the state $x$), by writing the output as a convolution of the input with a certain kernel. For an introduction to the terminology of realization theory and input-output systems in finite dimension, we refer to the textbooks of Polderman and Willems \cite{polderman1998introduction} or Sontag \cite{sontag}.

Yamamoto's approach to the realization theory for infinite-dimensional systems considers systems in which the input $u$ is applied during a time interval of the form $[-T, 0]$, with $T>0$ arbitrary, and the output is a certain function $t \mapsto y(t)$ defined for positive times, in a suitable functional space. In order to relate System \eqref{system_lin_formel2} with Yamamoto's realization theory, we rewrite it, after suitable time translations, as

\begin{equation}
\label{syst_lin_avec_sortie}
\begin{dcases}
 x(t)=\sum_{j=1}^NA_jx(t-\Lambda_j)+Bu(t),&\text{ for
  $t\ge \inf \supp(u)$}, \\
  x(t)=0,&\text{ for
  $t<\inf \supp(u)$},\\
y(t)=x(t-\Lambda_N),&\text{ for
$t \in [0,+\infty)$},
\end{dcases}
\end{equation}
where the input $u$ belongs to  
\begin{equation*}
\Omega_{q}=\{u\in L^q
(\mathbb{R}_-,\mathbb{R}^m) \mid 
\text{$\supp(u)$ is compact}
\},
\end{equation*}
with $\supp(u)$ denoting the support of $u$.

We aim at writing System~\eqref{syst_lin_avec_sortie} as a convolution operator with a kernel in the space of Radon measures, i.e., we want to find $A \in M_+(\mathbb{R})$ such that the input-output system~\eqref{syst_lin_avec_sortie} can be represented as
\begin{equation}
\label{representation_radon_measure1}
y(t) = (A * u)(t) = \int_{-\infty}^{+\infty} d A(\tau) u(t - \tau), \qquad \mbox{$t\in [0,+\infty)$}.
\end{equation} 
Note that the convolution of a $d \times m$ matrix-valued Radon measure with a compactly supported function in $L^q(\mathbb{R},\mathbb{R}^m)$ belongs to $L^q(\mathbb{R},\mathbb{R}^d)$. Recalling that $\pi$ is the truncation operator on positive times defined in Equation~\eqref{eq:pi}, we can rewrite Equation~\eqref{representation_radon_measure1} in a more convenient way as
\begin{equation}
\label{representation_radon_measure1bis}
y=\pi(A*u).
\end{equation}
To achieve the goal of finding a Radon measure $A$ satisfying Equation~\eqref{representation_radon_measure1bis}, we define the zero-order distributions
\begin{align}
\label{eq:defQ}
Q & :=\delta_{-\Lambda_N} I_d- \sum_{j=1}^N \delta_{-\Lambda_N+\Lambda_j} A_j,\\
\label{eq:defP}
P & := B \delta_0.
\end{align}
The  matrix-valued distributions $Q$ and $P$ are in $M(\mathbb{R}_-)$, the space of Radon measures with compact support included in $\mathbb{R}_-$. These distributions are naturally associated with System~\eqref{syst_lin_avec_sortie} for two major reasons.

We define the state space of System~\eqref{syst_lin_avec_sortie} in terms of the distribution $Q$ as
 \begin{equation}
 \label{def_XQ}
 X^{Q,\,q}:=\left\{y \in L^q_{\rm loc}\left(\mathbb{R}_+,\mathbb{R}^d\right)\suchthat \pi(Q*y)=0  \right\}.
 \end{equation}
The system thus has an input $u$ belonging to $\Omega_q$ and an output $y$ in $ X^{Q,\,q}$. We remark that the set $X^{Q,\,q}$ can be easily identified with the space $L^q\left(\left[0,\Lambda_N\right],\mathbb{R}^d\right)$. In fact, $y \in  X^{Q,\,q}$ if and only if the restriction $y|_{[0,\Lambda_N]}$ is in $L^q([0,\Lambda_N],\mathbb{R}^d)$ and $y$ is the unique extension of $y|_{[0,\Lambda_N]}$ on the interval $[0,+\infty)$ satisfying the condition $\pi(Q*y)=0$.

The distributions $Q$ and $P$ allow us to obtain the Radon measure $A$ representing System~\eqref{syst_lin_avec_sortie} as a convolution operator. Denoting by $\tilde{y}$ the natural extension of the output $y$ on $\mathbb{R}$, i.e., $\tilde{y}(t)=x(t-\Lambda_N)$ for $t \in \mathbb{R}$, the first equation of \eqref{syst_lin_avec_sortie} implies that
\begin{equation}
\label{obtention_convolution1}
\left(Q*\tilde{y}\right)(t)=\left(P*u\right)(t),\quad t \in \mathbb{R}.
\end{equation}
The distribution $Q$ is invertible over $\mathcal{D}'_+(\mathbb{R})$ in convolution sense and the inverse distribution $Q^{-1}$ belongs to $M_+(\mathbb{R})$. More precisely, the distribution
\[
Q^{-1}:=\delta_{\Lambda_N}*  \left( \sum_{n=0}^{+\infty} \left(   \sum_{j=1}^N \delta_{\Lambda_j} A_j \right)^n \right)
\]
is easily seen to be the inverse of $Q$ by a Neumann series argument. We take the convolution product of Equation~\eqref{obtention_convolution1} on the left by $Q^{-1}$ and we obtain
\begin{equation}
\label{obtention_convolution2}
\tilde{y}(t)=\left(Q^{-1}*P*u\right)(t),\quad t \in \mathbb{R}.
\end{equation}
Applying the operator $\pi$ in Equation~\eqref{obtention_convolution2}, we have
\begin{equation}
\label{representation_radon_measure3}
y(\cdot)=\pi \left(A*u\right)(\cdot),\qquad \text{where } A:=Q^{-1}*P.
\end{equation}
Notice also that the measure $A$ belongs to $M_+(\mathbb{R})$. As a consequence, System~\eqref{syst_lin_avec_sortie} is pseudo-rational in the sense of Yamamoto (see, e.g., \cite{yamamoto1989reachability}).

\begin{remark}
Equation~\eqref{representation_radon_measure1} provides an expression for the relation between the input $u$ and the output $y$ of System~\eqref{syst_lin_avec_sortie}. Other expressions for this same relation can be obtained by other means, such as by relying on the operator $E_{q}(T)$ as done in Section~\ref{sec:6} or by expressing $y$ as a Stieltjes integral of $u$ with respect to the fundamental solution of the system, as done in the variation-of-constants formula in \cite{Hale}. While clearly equivalent, some representations can be more suitable than others for a given purpose. The representation through the operator $E_{q}(T)$ was useful in Section~\ref{sec:6} to prove Theorem~\ref{lem:3}, and in the remaining part of the paper we will use \eqref{representation_radon_measure1} to obtain the controllability results from Theorems~\ref{main_result1} and \ref{main_result2}.
\end{remark}

We now characterize the controllability notions from Definition~\ref{def:contr-from-origin} in terms of the above realization theory formalism.
 
\begin{proposition}
\label{def_reachability}
System~\eqref{system_lin_formel2} is
\begin{enumerate}[i)]
\item $L^q$ approximately controllable if and only if for every $\phi \in  X^{Q,\,q}$ there exists a sequence of inputs $(u_n)_{n \in \mathbb{N}} \in (\Omega_q)^{\mathbb{N}}$ such that its associated sequence of outputs $(y_n)_{n \in \mathbb{N}} \in \left(L^q_{\rm loc}\left(\mathbb{R}_+,\mathbb{R}^d\right) \right)^{\mathbb{N}}$ through System~\eqref{syst_lin_avec_sortie} satisfies
\[
y_n \underset{n \to +\infty}{\longrightarrow} \phi \quad \text{in} \quad L^q_{\rm loc}\left(\mathbb{R}_+,\mathbb{R}^d\right);
\]
\item $L^q$ exactly controllable if and only if for every $\phi \in  X^{Q,\,q}$ there exists $u \in \Omega$ such that 
its associated output through System~\eqref{syst_lin_avec_sortie} satisfies
\[
y(\cdot)= \phi(\cdot).
\]
\end{enumerate}
\end{proposition}

\begin{remark}
Following Yamamoto (see, e.g., \cite{yamamoto1989reachability}), we will refer in the sequel to the characterization of approximate and exact controllability of System~\eqref{system_lin_formel2} given in Proposition~\ref{def_reachability} as \emph{approximate and exact controllability of System~\eqref{syst_lin_avec_sortie}}. Note also that, in \cite{yamamoto1989reachability}, approximate (respectively, exact) controllability of System~\eqref{system_lin_formel2} is also called \emph{quasi-rea\-cha\-bi\-lity} (respectively, \emph{reachability}).
\end{remark}

\subsection{Preliminary properties and proofs of Propositions~\ref{Prop2:main_result0} and \ref{Prop2:main_result}}
\label{subsec:useful_properties}

Before making use of the realization theory to provide criteria for approximate and exact controllabilities in Sections~\ref{subsec:Approximate_contr_HY} and \ref{subsec:Exact_contr_HY}, respectively, we provide in this section preliminary technical results relating properties of the (Laplace transform of the) distributions $Q$ and $P$ introduced in the framework of realization theory with $H$ associated with System~\eqref{system_lin_formel2}. In particular, we shall also obtain from those technical results a proof of Propositions~\ref{Prop2:main_result0} and \ref{Prop2:main_result}.

We start by considering the Laplace transforms of the distributions $Q$ and $P$, which are given by (see formulas \eqref{laplace_transform_radon_measure} and \eqref{laplace_transform_dirac})
\begin{equation*}
\widehat{Q}(p)=e^{p \Lambda_{N}}I_d- \sum\limits_{j=1}^N e^{ p(\Lambda_N- \Lambda_j)} A_j,\quad \widehat{P}(p)\equiv \widehat{P}=B, \qquad p \in \mathbb{C}.
\end{equation*}
In the following we  use in an equivalent way $B$ and $\widehat{P}$ depending on the context. Actually, when referring to Yamamoto's articles, it is more practical to use $\widehat{P}$, while it is better to use $B$ when we deal with Hautus--Yamamoto criteria for System~\eqref{system_lin_formel2}.

We first notice that $H(\cdot)$ and $\widehat{Q}(\cdot)$ satisfy the relation
\begin{equation}
\label{lien_entre_les_deux_fonctions}
H(p)=e^{-p \Lambda_N} \widehat{Q}(p),\qquad p \in \mathbb{C}.
\end{equation}
In Proposition~\ref{prop:lienQH} given below, we link the rank condition of the operator $H(\cdot)$ associated with System~\eqref{system_lin_formel2} with the rank condition of the operator $\widehat{Q}(\cdot)$ associated with System~\eqref{syst_lin_avec_sortie}. As a preliminary step, let us prove the following technical lemma.

\begin{lemma}
\label{lem:sous_suite1}
With the notations introduced above, the following properties hold true:
\begin{enumerate}[i)]
\item\label{item:suites-1} There exists $M\in \overline{H(\mathbb{C})}$ such that $\rank\left[M,B\right]<d$ if and only if there exist $g \in \mathbb{C}^d$ and $(p_n)_{n \in \mathbb{N}} \in \mathbb{C}^{\mathbb{N}}$ with bounded real part such that $\|g\|=1$, $g^TB=0$, and $\lim_{n\to\infty}g^TH(p_n)=0$.

\item\label{item:suites-2} Under the assumption that $\rank\left[A_N,B\right]=d$, there exists $\widetilde{Q}\in \overline{\widehat{Q}(\mathbb{C})}$ such that
$\rank\bigl[\widetilde{Q},\allowbreak B\bigr] \allowbreak <d$ if and only if there exist $g \in \mathbb{C}^d$ and $(p_n)_{n \in \mathbb{N}} \in \mathbb{C}^{\mathbb{N}}$ with bounded real part such that $\|g\|=1$, $g^TB=0$, and $\lim_{n\to\infty} g^T\widehat{Q}(p_n) = 0$.

\item\label{item:suites-3} Under the assumption that $\rank\left[A_N,B\right]=d$, Condition \ref{exa-b} in Proposition~\ref{Prop2:main_result} is not satisfied if and only if there exist $g\in \mathbb{C}^d$ and $(p_n)_{n \in \mathbb N} \in \mathbb{C}^{\mathbb{N}}$ with bounded real part such that $\|g\|=1$, $g^T B=0$, and $\lim_{n\to\infty} g^T H(p_n) = 0$.
\end{enumerate} 
\end{lemma}

\begin{proof}
We start by proving Item~\ref{item:suites-1}.  Let $g\in \mathbb{C}^d$ and $(p_n)_{n\in \mathbb{N}}$ with bounded real part be such that $\|g\|=1$, $g^T B=0$ and $\lim_{n\to\infty}g^T H(p_n) =0$. Since $H(\cdot)$ is uniformly bounded on any bounded vertical strip, it follows that, up to a subsequence, $H(p_n)$ converges to some matrix $M \in \mathcal{M}_{d,d}(\mathbb{C})$ as $n\to\infty$. We deduce that $g^T M=0$ and $g^T B=0$, proving that $\rank\left[M,B\right]<d$. Conversely, assume that $g \in \mathbb{C}^d$ and $M$ are such that $\|g\|=1$, $g^TB=0$, $g^TM=0$, and $M=\lim_{n\to\infty}H(p_n)$ for some sequence  $(p_n)_{n \in \mathbb{N}} \in \mathbb{C}^{\mathbb{N}}$. The sequence $(p_n)_{n\in \mathbb{N}}$ has bounded real part because of the following properties of $H$: $H(p)$ is nonsingular for $p$ out of a bounded vertical strip (see Proposition~\ref{prop_zero_Q0}), $H(p)$ converges to $I_d$ when the real part of $p$ tends to $+\infty$, and $H(p)$ diverges in norm when the real part of $p$ tends to $-\infty$ (cf.~the proof of Proposition~\ref{prop_zero_Q0}). This concludes the proof of the converse implication. 

The proof of Item~\ref{item:suites-2} is similar to that of Item~\ref{item:suites-1}, with $\widehat Q$ playing the role of $H$. The only difference is that, in the second part of the argument, one needs to use the assumption that $\rank\left[A_N,B\right]=d$ in order to ensure that the sequence $(p_n)_{n \in \mathbb{N}} \in \mathbb{C}^{\mathbb{N}}$ has bounded real part, since $\widehat Q(p)$ converges to $-A_N$ (instead of $I_d$) when the real part of $p$ tends to $-\infty$ (instead of $+\infty$).

Let us prove Item~\ref{item:suites-3}. If  condition \ref{exa-b} in Proposition~\ref{Prop2:main_result}
is not satisfied, there exist a sequence $(p_n)_{n \in \mathbb{N}} \in \mathbb{C}^{\mathbb{N}}$ and a sequence of vectors $(g_n)_{n \in \mathbb{N}} \in \left(\mathbb{C}^d\right)^{\mathbb{N}}$ such that $\|g_n^T\|=1$  for all $n \in \mathbb{N}$,  $\lim_{n\to\infty} g_n^TB = 0$, and $\lim_{n\to\infty} g_n^T H(p_n) = 0$. Proposition~\ref{prop_zero_Q0} and $\rank\left[A_N,B \right]=d$ imply that the real part of $(p_n)_{n \in \mathbb{N}}$ is bounded, and thus $(H(p_n))_{n \in \mathbb N}$ is also bounded. Without loss of generality, $\lim_{n\to\infty} g_n = g$ for some vector $g \in \mathbb{C}^d$ and we deduce that $\|g\|=1$, $g^TB=0$, and  $\lim_{n\to\infty}g^T H(p_n) = 0$, which proves one of the two implications. The proof of the converse is obvious.
\end{proof}

Thanks to Lemma~\ref{lem:sous_suite1}, we are able to prove the following proposition establishing a link between the rank properties of $H(\cdot)$ and $\widehat{Q}(\cdot)$.

\begin{proposition}
\label{prop:lienQH}
The following two properties hold true:
\begin{enumerate}[i)]
\item \label{prop:lienHQ1} $\rank\left[H(p),B\right]=d$ for every $p \in \mathbb{C}$ if and only if $\rank\left[\widehat{Q}(p),B\right]=d$ for every $p \in \mathbb{C}$, 
\item \label{prop:lienHQ2} Under the assumption that $\rank\left[A_N,B\right]=d$, we have that $\rank\left[M,B\right]=d$  for every $M\in \overline{H(\mathbb{C})}$ if and only if $\rank\left[ \widetilde{Q},B\right]=d$ for every $\widetilde{Q}\in \overline{\widehat{Q}(\mathbb{C})}$.
\end{enumerate}
\end{proposition}    

\begin{proof}
Item~\ref{prop:lienHQ1} is a straightforward consequence of Equation~\eqref{lien_entre_les_deux_fonctions}. Let us now prove Item~\ref{prop:lienHQ2}. Note that, thanks to Lemma~\ref{lem:sous_suite1}, there exists $M\in \overline{H(\mathbb{C})}$ such that  $\rank\left[M,B\right]<d$ if and only if there exists
$g \in \mathbb{C}^d$ and 
$(p_n)_{n \in \mathbb{N}} \in \mathbb{C}^{\mathbb{N}}$ with bounded real part such that $\|g\|=1$, $g^TB=0$, and $\lim_{n\to\infty}g^TH(p_n) = 0$. By Equation~\eqref{lien_entre_les_deux_fonctions}, we can rewrite the last relation as $\lim_{n\to\infty}g^Te^{-p_n \Lambda_N}\widehat{Q}(p_n) = 0$, which is equivalent to $\lim_{n\to\infty}g^T\widehat{Q}(p_n) = 0$, since the real part of $(p_n)_{n \in \mathbb{N}} \in \mathbb{C}^{\mathbb{N}}$ is bounded.
\end{proof}

A simple adaptation of Lemma~\ref{lem:sous_suite1}, giving an analogue of Item~\ref{item:suites-3} in terms of $\widehat Q$, allows us to prove the following proposition.

\begin{proposition}
\label{prop:eq_Q_corona_thm}
Under the assumption that $\rank\left[A_N,B\right]=d$, the following statements are equivalent:
\begin{enumerate}[(a)]
\item\label{corona-a} $\rank\left[\widetilde{Q},B\right]=d$ for every $\widetilde{Q}\in \overline{\widehat{Q}(\mathbb{C})}$,

\item \label{corona-b} there exists $\alpha>0$ such that, for every $p \in \mathbb{C}$,
\[
\inf \left\{ \vertii{g^T \widehat{Q}(p) }+\vertii{g^T B}\suchthat  \mbox{$g\in \mathbb{C}^d$, $\|g^T \|=1$} \right\} \ge \alpha. 
\]
\end{enumerate}
\end{proposition}    

We conclude this section by providing the proofs of Propositions~\ref{Prop2:main_result0} and \ref{Prop2:main_result}.

\begin{proof}[Proof of Propositions~\ref{Prop2:main_result0} and~\ref{Prop2:main_result}]
We just prove Proposition~\ref{Prop2:main_result} because the proof of Proposition~\ref{Prop2:main_result0} can be obtained following similar arguments, based on a simpler version of Lemma~\ref{lem:sous_suite1}, adapted to the case of approximate controllability.

The equivalence between \ref{exa-a} and \ref{exa-b} is a consequence of Items~\ref{item:suites-1} and \ref{item:suites-3} of Lemma~\ref{lem:sous_suite1}. Let us now prove that \ref{exa-b'} implies \ref{exa-b}. If \ref{exa-b} is not satisfied, it follows by Lemma~\ref{lem:sous_suite1} that there exist  $g\in \mathbb{C}^d$ and a sequence $(p_n)_{n \in \mathbb{N}}\in \mathbb{C}^{\mathbb{N}}$ with a bounded real part such that $\|g\|=1$, $g^T B=0$, and $\lim_{n\to\infty} g^T H(p_n) = 0$. Since $H(\cdot)$ is uniformly bounded on every bounded vertical strip of $\mathbb{C}$, we deduce that condition \ref{exa-b'} is not satisfied. Thus \ref{exa-b'} implies \ref{exa-b}. The converse implication is obvious.
\end{proof}

\subsection{Approximate controllability}    
\label{subsec:Approximate_contr_HY}

The approximate controllability criterion for the approximate controllability of System~\eqref{syst_lin_avec_sortie} is an application of the 
paper \cite{yamamoto1989reachability} by Yamamoto.

\begin{theorem}
\label{th_Yamamoto1}
Let $q \in [1,+\infty)$. System \eqref{syst_lin_avec_sortie} is $L^q$ approximate controllable if and only if the following conditions hold true:
\begin{enumerate}[i)]
\item \label{assumption1-Lq}$\rank\left[\widehat{Q}(p),B\right]=d$ for every $p \in \mathbb{C}$,
\item \label{assumption2-Lq} $\rank\left[A_N,B\right]=d$.
\end{enumerate}
\end{theorem}

\begin{proof}[Proof\ \footnotemark.]
Equivalence between $L^2$ approximate controllability and Items~\ref{assumption1-Lq} and \ref{assumption2-Lq} follows from \cite[Corollary~4.10]{yamamoto1989reachability} after recalling that $P = B \delta_0$ and noting that $Q$ can be written as $Q = Q_0 + Q_1$, with $Q_0 = \delta_0 A_N$ and $\supp Q_1 \subset [-\Lambda_N, -\Lambda_N + \Lambda_{N-1}]$, which is bounded away from zero.

\footnotetext{This proof differs from the one provided in the published version of this paper, as the argument provided in the latter was incomplete.}

It was further shown in \cite[Theorem~4.4]{YamamotoRealization} that $L^2$ approximate controllability is equivalent to the existence of sequences of matrices $(R_n)_{n \in \mathbb N}$ and $(S_n)_{n \in \mathbb N}$ in $\mathcal E^\prime(\mathbb R_-)$ of appropriate sizes such that one has the approximate Bézout's identity
\begin{equation*}
\lim_{n \to +\infty} Q * R_n + P * S_n = \delta_0 I_d \qquad \text{ in } \mathcal D^\prime(\mathbb R).
\end{equation*}
The proof of this equivalence in \cite[Theorem~4.4]{YamamotoRealization} can be immediately generalized in order to also show equivalence between $L^q$ approximate controllability and the above approximate Bézout's identity, for any $q \in [1, +\infty)$. In particular, we deduce that $L^q$ and $L^2$ approximate controllabilities are equivalent, yielding the conclusion.
\end{proof}

We are finally in position to provide a proof of Theorem~\ref{main_result1}.

\begin{proof}[Proof of Theorem~\ref{main_result1}]
Theorem~\ref{lem:3} proves that $L^q$ approximate controllability is equivalent to the $L^q$ approximate controllability in time $T=d \Lambda_N$. The conclusion follows by combining Theorem~\ref{th_Yamamoto1} and Item~\ref{prop:lienHQ1} of Proposition~\ref{prop:lienQH}.
\end{proof}

\subsection{Exact controllability}
\label{subsec:Exact_contr_HY}

In order to take advantage of realization theory, one needs to expound and improve some of the exact controllability results obtained by Yamamoto.

\subsubsection{Bézout's identity characterization of Radon exact controllability}

As explained in the introduction, Yamamoto's realization approach   
tackles exact controllability in 
distributional sense. Thanks to the specific features of our difference delay system, we are able to replace the
general distributional framework of Yamamoto by the more structured setting of distributions of order zero (or, equivalently, Radon measures). More precisely,
we first define the  Radon measure space
\begin{equation*}
 \overline{X}^Q:=\left\{ \pi \Psi\suchthat\Psi \in  \left(M(\mathbb{R}_+)\right)^d \mbox{ and } \pi(Q*\pi \Psi)=0  \right\}.
 \end{equation*}
It is then straightforward to see that outputs of the input-output system defined in Equation~\eqref{representation_radon_measure3} corresponding to inputs in 
$\left(M(\mathbb{R}_-)\right)^{m}$ belong to $\overline{X}^Q$.
We next extend the definition of exact controllability to Radon measures as follows.

\begin{definition}
\label{def_reachability_distribution}
System~\eqref{syst_lin_avec_sortie} is \emph{Radon exactly controllable} if, for every $\pi \Psi \in \overline{X}^Q$ with $\Psi \in (M(\mathbb{R}_+))^d$, there exists $u \in \left(M(\mathbb{R}_-)\right)^m$ such that $\pi(A*u)=\pi \Psi$.
\end{definition}

In Proposition~\ref{prop_fond_exact_controllability} below, we give a characterization of Radon exact controllability through a Bézout identity over the ring of Radon measures. We then discuss the link between such a Bézout identity and $L^q$ exact controllability in Corollary~\ref{cor:distru_exac_equiv}. Let us start by summarizing Lemmas~4.3, A.2, and A.3 from \cite{YamamotoRealization} in the following statement.

\begin{lemma}
\label{lem:sumlemmasYamamoto}
The following assertions hold true:
\begin{enumerate}[i)]
\item \label{lem:4.3_Yamamoto}
Let $\Psi$ be an element in $ \mathcal{D}'_{+}(\mathbb{R})$ such that $\pi(Q*\pi \Psi)=0$. Then there
exists a sequence $\Psi_n \in X^{Q,2}$ such that $ \lim_{n\to\infty} \Psi_n = \pi \Psi$.
\item  \label{lem:A_2_Yamamoto}
For $ \alpha \in \mathcal{D}'(\mathbb{R})$, we have $\pi(\alpha)=0$ if and only if $\supp(\alpha) \subset (-\infty,0]$.
\item \label{lem:A_3_Yamamoto}
$\pi(\alpha * \pi \beta)=\pi(\alpha*\beta)$ for every $\alpha \in \mathcal{E}'(\mathbb{R}_{-})$, $\beta \in \mathcal{D}'_{+}(\mathbb{R})$.
\end{enumerate}
\end{lemma}

\begin{proposition}
\label{prop_fond_exact_controllability}
System~\eqref{syst_lin_avec_sortie} is Radon exactly controllable if and only if there exist  two matrices $R$ and $S$ with entries in $M(\mathbb{R}_{-})$ such that
\begin{equation}
\label{eq:exact_cont0}
Q*R+P*S=\delta_0 I_d.
\end{equation}
Moreover, if \eqref{eq:exact_cont0} holds true, then,
for every target output $\pi \Psi \in \overline{X}^Q$ with $\Psi \in (M(\mathbb{R}_+))^d$, the input $\omega \in M_+(\mathbb R)$ given by
\begin{equation}\label{eq:motion-planning}
\omega = S * Q * \Psi
\end{equation}
steers the origin to the state $\pi \Psi$
along the system \eqref{syst_lin_avec_sortie}, i.e.,
\begin{equation}\label{eq:sol-con-B}
    \pi(A * \omega) = \pi \Psi.
\end{equation}
\end{proposition}

\begin{proof}
We first prove that \eqref{eq:exact_cont0} is a necessary condition for Radon exact controllability of \eqref{syst_lin_avec_sortie}. The entries of $Q^{-1}$ are in the space $M_{+}(\mathbb{R})$ and we have $\pi (Q^{-1})=Q^{-1}$ so that the columns of $Q^{-1}$ are in $\overline{X}^Q$. Since System~\eqref{syst_lin_avec_sortie} is Radon exactly controllable, we can take $\Psi$ equal to each column of $Q^{-1}$ in Definition~\ref{def_reachability_distribution} and thus deduce the existence of a matrix $S$ of size $m \times d$ and with entries in $M(\mathbb{R}_-)$ such that
\begin{equation}
\label{eq:bezout20}
\pi(Q^{-1}*P*S)
=\pi \left(Q^{-1}\right).
\end{equation}
We define
\begin{equation}
\label{eq:bezout30}
R:= Q^{-1}-Q^{-1}*P*S.
\end{equation}
From Equation~\eqref{eq:bezout20}, we deduce that $\pi(R)=0$ and  Item~\ref{lem:A_2_Yamamoto} of Lemma~\ref{lem:sumlemmasYamamoto} proves that the support of $R$ is included in 
$\mathbb{R}_-$. Since the supports of $Q^{-1}$, $P$, and $S$ are bounded on the left, the same is true for the support of $R$, which is, therefore, compact. Moreover the entries of $R$ are distributions of order zero because the same  is true for $Q^{-1}$, $P$, and $S$. Hence the entries of $R$ belong to $M(\mathbb R_-)$, and we finally obtain \eqref{eq:exact_cont0} by taking the convolution product of \eqref{eq:bezout30} on the left by $Q$.

We now prove that Condition~\eqref{eq:exact_cont0} is sufficient for Radon exact controllability of \eqref{syst_lin_avec_sortie}. Let $R$ and $S$ be two  matrices with entries in the space $M(\mathbb{R}_{-})$ satisfying Equation~\eqref{eq:exact_cont0}.
Let $\Psi \in \left(M(\mathbb{R}_+)\right)^d$ be such that
\begin{equation}\label{eq:psiXQ}
\pi(Q*\pi\Psi)=0,
\end{equation}
so that 
$\pi\Psi \in \overline{X}^Q$. 
Set $\omega:=S*Q*\Psi$. Since $Q$ and $S$ have entries in $M(\mathbb{R_-})$, we have that $\omega $ belongs to the space $M_+(\mathbb{R})$. Item~\ref{lem:A_3_Yamamoto} of Lemma~\ref{lem:sumlemmasYamamoto} implies that 
\begin{equation*}
\pi\left( \omega\right)= \pi\left(S*\pi\left(Q*  \pi\Psi\right)\right)=0,
\end{equation*}
and we deduce that $\omega$ is in $\left(M(\mathbb{R}_-)\right)^{m}$ by Item~\ref{lem:A_2_Yamamoto} of Lemma~\ref{lem:sumlemmasYamamoto}. By definition of $\omega$ and Equation~\eqref{eq:exact_cont0}, we also have
\begin{align}
\label{eq:bezout5}
Q^{-1}*P*\omega+R*Q*\Psi&=
Q^{-1}*(P*S+Q*R)*Q*\Psi=
\Psi.
\end{align}
We deduce from Equation~\eqref{eq:bezout5}, Condition~\ref{lem:A_3_Yamamoto} of Lemma~\ref{lem:sumlemmasYamamoto}, and 
\eqref{representation_radon_measure3}
that
\begin{align*}
\pi \Psi & = \pi\left(Q^{-1}*P*\omega\right)+\pi\left(R*Q*\Psi\right)\\
& =\pi\left(A*\omega\right)+\pi\left(R*\pi\left(Q*\pi \Psi \right)\right)\\
& = \pi\left(A*\omega\right),
\end{align*} 
where the last equality follows from \eqref{eq:psiXQ}.
We have shown that $\pi \Psi$ is the output corresponding to $\omega$ for the input-output map~\eqref{representation_radon_measure3}, proving the Radon exact controllability of \eqref{syst_lin_avec_sortie}.
\end{proof}

\begin{remark}
Proposition~\ref{prop_fond_exact_controllability} can be seen as a Radon counterpart of the distributional Bézout identity characterization of exact distributional controllability provided in \cite{Yamamoto_Willems}.
\end{remark}

As a consequence of Proposition~\ref{prop_fond_exact_controllability}, we can now deduce that Radon exact controllability implies $L^q$ exact controllability.

\begin{corollary}
\label{cor:distru_exac_equiv}
If System~\eqref{syst_lin_avec_sortie} is Radon exactly controllable then it is $L^q$ exactly controllable for every $q \in [1,+\infty)$.
\end{corollary}

\begin{proof}
Let $R$ and $S$ be two  matrices with entries belonging to the space $M(\mathbb{R}_{-})$ satisfying Equation~\eqref{eq:exact_cont0}. Let $q \in [1,+\infty)$. Let $y \in  X^{Q,\,q}$ and denote by $\tilde{y}$ the extension of $y$ on $(-\infty,+\infty)$ by setting $\tilde{y}$ equal to zero on $(-\infty,0)$. Set $\omega=S*Q*\tilde{y}$ as in \eqref{eq:motion-planning}. Then $\omega$ is in $\Omega_q$ and it follows from \eqref{eq:sol-con-B} that $y=\pi\left(A*\omega\right)$, which proves $L^q$ exact controllability thanks to Proposition~\ref{def_reachability}.
\end{proof}

A challenging question is to investigate a converse to the previous corollary, i.e., whether $L^q$ exact controllability for some $q \in [1,+\infty)$ implies Radon exact controllability or not. We bring a positive answer to that question for $q=1$ in the following section.

\subsubsection{Bézout's identity characterization of \texorpdfstring{$L^1$}{L1} exact controllability}

We next provide a sufficient and necessary condition for System~\eqref{syst_lin_avec_sortie} to be $L^1$ exactly controllable.

\begin{theorem}
\label{prop_fond_exact_controllability_L1}
System~\eqref{syst_lin_avec_sortie} is $L^1$ exactly controllable if and only if there exist  two matrices $R$ and $S$ with entries in $M(\mathbb{R}_{-})$ such that \eqref{eq:exact_cont0} holds true.
\end{theorem}

\begin{proof}
We first notice that, by Proposition~\ref{prop_fond_exact_controllability} and Corollary~\ref{cor:distru_exac_equiv}, Condition~\eqref{eq:exact_cont0} implies that System~\eqref{syst_lin_avec_sortie} is $L^1$ exactly controllable.

Let us now prove the converse implication. We proceed in four steps.

\begin{enumerate}[label={\textbf{Step~\arabic*.}}, ref={Step~\arabic*}, itemindent=*, leftmargin=0pt]
\item Let us define the map $\widetilde{G}:\, \widetilde{\Omega}_1 \longrightarrow L^1([0,\Lambda_N],\mathbb{R}^d)$ by
\[\widetilde{G}(u)(t)=\pi(Q^{-1}*P*u)(t),\qquad t\in [0,\Lambda_N],\;u \in \widetilde{\Omega}_1,\]
where $\widetilde{\Omega}_1$ denotes the subspace of $\Omega_1$ made of inputs with compact support in $[-d \Lambda_N,0]$, endowed with the norm $\|.\|_{[-d\Lambda_N,0],1}$. Firstly, we can see that the map $\widetilde{G}$ is a bounded operator because $Q^{-1}*P$ is a distribution with a finite number of Dirac distributions on each compact interval of $\mathbb{R}$. We deduce that $\widetilde{G}$ is a continuous linear map. Secondly, the saturation of the endpoint map (Theorem~\ref{lem:RanE_con}) allows us to state that System~\eqref{syst_lin_avec_sortie} is $L^1$ exactly controllable if and only if the map $\widetilde{G}$ is surjective. We can now apply 
the open mapping theorem (see, e.g, \cite[Theorem~4.13]{rudin1991functional}) and deduce that there exists $\delta>0$ such that 
\begin{equation}
\label{eq:open_map_thereom}
\widetilde{G}(U) \supset \delta V,
\end{equation}
where $U$ and $V$ are the open unit balls of $ \widetilde{\Omega}_1$ and $L^1([0,\Lambda_N],\mathbb{R}^d)$ respectively. 

\item\label{bezout-L1-step2} Since $\pi(Q*\pi (Q^{-1}))=\pi(\delta_0)=0$ and the inclusion $X^{Q,2} \subset X^{Q,1}$ holds, Item~\ref{lem:4.3_Yamamoto} of Lemma~\ref{lem:sumlemmasYamamoto} implies that there exists a sequence $\psi_n=(\psi_{n,1},\dots,\psi_{n,d}) \in (X^{Q,1})^d$,  $n \in \mathbb{N}$, such that $ \psi_n \rightarrow \pi (Q^{-1})$ in the distributional sense as $n\to\infty$. In other words, for $i,j \in \{1,\dots,d\}$, if we define the Radon measures $\left(Q^{-1}_{n}\right)_{i,j}(t)=\int_0^t \left(\psi_{n}\right)_{i,j}(x)dx$ for $t \in \mathbb{R}_+$ and $n \in \mathbb{N}$, we get that $\left(Q^{-1}_n\right)_{i,j}$ weak-star converges to $\left(\pi(Q^{-1})\right)_{i,j}$ in the sense of \cite[Paragraph 4.3]{maggi2012sets}. By \cite[Remark 4.35]{maggi2012sets} and the Banach--Steinhaus theorem, we obtain that the total variation 
\begin{equation*}
\underset{n \in \mathbb{N}}{\sup}\abs*{\left(Q^{-1}_n\right)_{i,j}}([0,\Lambda_N])<\infty,
\end{equation*}
where 
\begin{equation*}
\abs*{\left(Q^{-1}_n\right)_{i,j}}([0,\Lambda_N]):=\sup_{\substack{\varphi \in C([0, \Lambda_N]) \\ \vertii{\varphi}_{C([0, \Lambda_N])} \leq 1}} \abs*{\int_0^{\Lambda_N}  \varphi(t) d \left(Q^{-1}_n\right)_{i,j}(t)},
\end{equation*}
with $C([0, \Lambda_N])$ the space of the continuous functions defined on the interval $[0, \Lambda_N]$ with values in $\mathbb{R}$ endowed with its natural norm $\vertii{\cdot}_{C([0, \Lambda_N])}$ By the Riesz representation theorem, see \cite[Theorem~6.19]{Rudin1987Real}, we have that $|\left(Q^{-1}_n\right)_{i,j}|([0,\Lambda_N])$ denotes the total variation of the measure on the interval $[0,\Lambda_N]$ of $\left(Q^{-1}_n\right)_{i,j}$ in the sense of \cite[Chapter~6, Equation~(3)]{Rudin1987Real}. In particular, we have that 
\begin{equation*}
\abs*{\left(Q^{-1}_n\right)_{i,j}}([0,\Lambda_N])= \int_0^{\Lambda_N} \abs*{\left(\psi_n\right)_{i,j}(t)} dt,
\end{equation*}
so that each column of $\left(\psi_{n}\right)_{n \in \mathbb{N}} $  is uniformly bounded in $L^1([0,\Lambda_N],\mathbb{R}^d)$, that is,  there exists $C>0$ such that
\begin{equation*}
\|\psi_{n,j}\|_{L^1([0,\Lambda_N],\mathbb{R}^d)} \le C,\qquad \forall j \in \llbracket 1,d\rrbracket,\ \forall n \in \mathbb{N}.
\end{equation*}
Let $M'>0$ be such that $\delta M' >C$. We get from Equation~\eqref{eq:open_map_thereom} that
\begin{equation*}
\widetilde{G}(M'U) \supset \delta M'V
\end{equation*}
so that, for all $\psi_{n,j}$ with  $j \in \llbracket 1,d\rrbracket$ and $n \in \mathbb{N}$, there exists $S_{n,j} \in \widetilde{\Omega}_1$ such that
\begin{equation*}
\widetilde{G}(S_{n,j})= \psi_{n,j} \quad \mbox{and} \quad  \|S_{n,j}\|_{[-d\Lambda_N,0],\,1} \le M'.
\end{equation*}

\item We define $S_n=(S_{n,1},\dots,S_{n,d})$. By construction, $S_n \in {\widetilde{\Omega}_1}^d$ and 
\begin{equation}
\label{eq:bezout1}
\pi(Q^{-1}*P*S_n) \rightarrow \pi(Q^{-1}), \qquad \mbox{as $n\to\infty$},
\end{equation}
in a distributional sense. Since the columns of $S_n$, for $n \in\mathbb{N}$, are uniformly bounded for the norm in $\widetilde{\Omega}_1$, by the weak compactness of Radon measures (see for instance \cite[Theorem 4.33]{maggi2012sets}), there exist a matrix $S$ with entries in $M(\mathbb{R}_-)$ and a subsequence of $\left( S_n \right)_{n \in \mathbb{N}}$ (still denoted $\left(S_n\right)_{n \in \mathbb{N}}$ by abuse of notation) such that $\lim_{n \rightarrow + \infty}S_n = S$ in distributional sense. Since the convolution is continuous in distributional sense (see, for instance, \cite[Theorem 7.4.9]{bony2001cours}), and $\pi$ is continuous with respect to the strong dual topology, we deduce from Equation~\eqref{eq:bezout1} that
\begin{equation}
\label{eq:bezout2}
\pi(Q^{-1}*P*S)
=\pi \left(Q^{-1}\right).
\end{equation}
\item  We define
\begin{equation}
\label{eq:bezout3}
R:= Q^{-1}-Q^{-1}*P*S
\end{equation}
and we conclude the proof as in Proposition~\ref{prop_fond_exact_controllability}: 
by Equation~\eqref{eq:bezout2} we have $\pi(R)=0$, which, together with  Condition~\ref{lem:A_2_Yamamoto} 
of Lemma~\ref{lem:sumlemmasYamamoto}, implies that the support of $R$ is compact and contained in $(-\infty,0]$. Moreover, the entries of $R$ are distributions of order zero because the same  is true for $Q^{-1}$, $P$, and $S$. Bézout's identity is then obtained by multiplying Equation~\eqref{eq:bezout3} on the left by $Q$ in convolution sense. \qedhere
\end{enumerate}
\end{proof}

\begin{remark}
\label{remark1}
It is an open question whether, for each $q \in [1,+\infty)$, the Bézout identity is equivalent to the $L^q$ controllability. We can stress that the proof given for the case $q=1$ does not straightforwardly extend to $q > 1$. Indeed, in \ref{bezout-L1-step2}, the convergence $\psi_n \to \pi(Q^{-1})$ in the distributional sense implies the boundedness of $(\psi_n)_{n \in \mathbb N}$ in $L^1([0, \Lambda_N], \mathbb R^d)$, but such a sequence may fail to be bounded in $L^q([0, \Lambda_N], \mathbb R^d)$ for $q > 1$. 
\end{remark}

An immediate consequence of the Bézout identity characterization of the Radon and $L^1$ exact controllability is the following corollary.

\begin{corollary}
\label{corollaire_L^1_equiv_radon}
System~\eqref{syst_lin_avec_sortie} is Radon exactly controllable if and only if it is $L^1$ exactly controllable.
\end{corollary}

The open question raised in Remark~\ref{remark1} is tantamount to know if, for difference delay systems of the type~\eqref{system_lin_formel2}, the $L^q$ exact controllability for some $q \in [1,+\infty)$ is equivalent to the $L^q$ exact controllability  for every $q \in [1,+\infty)$. It has to be pointed out that such a property holds when dealing with the exponential stability of such systems, as noticed in  \cite{baratchart,Chitour2016Stability}.

\subsubsection{Solvability of Bézout's identity over Radon measure spaces and proof of Theorem~\ref{main_result2}}

Proposition~\ref{prop_fond_exact_controllability} and Theorem~\ref{prop_fond_exact_controllability_L1}  reduced the exact controllability problem to the 
problem of existence of solutions of a suitable Bézout identity. 

\begin{definition}
We say that \emph{the Bézout identity \eqref{eq:exact_cont0}
is solvable} if there exist two matrices $R$ and $S$ with entries in $M(\mathbb{R}_{-})$ that
satisfy it. 
\end{definition}

We next show that the solvability of the Bézout identity is equivalent to a corona problem, the latter having led a tremendous literature on the subject for some alike problems. We can cite for example the celebrated paper \cite{carleson1962interpolations} resolving the corona problem in one dimension for holomorphic bounded functions in the unit disk. However the present corona problem arising from the Bézout identity over a Radon measure algebra has not received a great attention and it is still an open question.
To the best of our knowledge, the closest result dealing with the solvability of a Bézout identity is that provided in  \cite[Theorem~5.1]{YAMAMOTO_Multi_Ring_2011} over the algebra of distributions, which gives a sufficient condition on $(Q,P)$ for such a solvability. In turn, solving a Bézout identity over Radon measures is fundamentally different than over distributions because of the different topologies endowing these two spaces.

Our first result for Radon measures is the following necessary condition for the solvability of the Bézout identity.

\begin{proposition}
\label{prop_nece_bezout}
A necessary condition for the Bézout identity \eqref{eq:exact_cont0} to be solvable is 
the following
\begin{enumerate}[i)]
\item 
\label{cdt_corona}

there exists $\alpha>0$ such that, for every $p\in \mathbb{C}$,  
\begin{equation}
\label{eq_corona}
\inf \left\{ \vertii{g^T\widehat{Q}(p)}+\vertii{g^TB} \suchthat \mbox{$g\in \mathbb{C}^d$, $\|g^T \|=1$} \right\} \ge \alpha.
\end{equation} 
\end{enumerate}
\end{proposition}

\begin{proof}
Let $R$ and $S$ be the two matrices with entries in $M(\mathbb R_-)$ satisfying Bezout's identity~\eqref{eq:exact_cont0}.  By contradiction, we assume that Condition~\ref{cdt_corona} is not satisfied, so that there exist a sequence $(p_n)_{n \in \mathbb{N}} \in \mathbb{C}^{\mathbb{N}}$ and a sequence $(g_n)_{n \in \mathbb{N}} \in \left(\mathbb{C}^d\right)^{\mathbb{N}}$ such that $\|g_n^T\|=1$  for all $n \in \mathbb{N}$,  $\lim_{n\to\infty} g_n^TB = 0$, and $\lim_{n\to\infty} g_n^T \widehat{Q}(p_n) = 0$.

We claim that there exists $\tilde{\alpha}>0$ such that $\Re(p_n)\le \tilde{\alpha}$ for every $n\in \mathbb{N}$. Indeed, $\widehat{Q}(p_n)$ is equivalent to $e^{p_n \Lambda_N} I_d$  when $n$ tends to $+\infty$, and unboundedness from above of $\Re(p_n)$ would contradict the relation $\lim_{n\to\infty} g_n^T \widehat{Q}(p_n) = 0$. We next get, by a classical estimate of the Laplace transform of an element of $M(\mathbb{R}_-)$, that there exists $C>0$ such that, for all $p \in \mathbb{C}$ with $\Re(p) \le \tilde{\alpha}$,
\begin{equation}
\label{eq:PW1}
\max \left\{\vertiii{\widehat{R}(p)},\,\vertiii{\widehat{S}(p)} \right\} \le C.
\end{equation}
We then deduce from Equation~\eqref{eq:exact_cont0} and the Laplace transform that
\begin{equation}
\label{eq:PW2}
g_n^T \widehat{Q}(p) \widehat{R}(p)+g_n^T B \widehat S(p) = g_n^T,\qquad n \in \mathbb{N},\ p\in\mathbb{C}.
\end{equation}

Equations~\eqref{eq:PW1} and \eqref{eq:PW2} imply that it is impossible to have both  $\lim_{n\to\infty} g_n^TB = 0$ and $\lim_{n\to\infty} g_n^T \widehat{Q}(p_n) = 0$. We reached a contradiction, so that Condition~\ref{cdt_corona} is a necessary condition for the solvability of Bézout's identity.
\end{proof}

We can now undertake the proof of Theorem~\ref{main_result2}.

\begin{proof}[Proof of Theorem~\ref{main_result2}]
Theorem~\ref{lem:3} states that $L^1$ exact controllability is equivalent to the $L^1$ exact controllability in time $T=d \Lambda_N$. By Theorem~\ref{main_result1}, Condition~\ref{assumption2-L1} of Theorem~\ref{main_result2} is necessary for $L^1$ approximate controllability, so that it is also necessary for $L^1$ exact controllability. Combining Theorem~\ref{prop_fond_exact_controllability_L1} and Proposition~\ref{prop_nece_bezout}, we deduce that \eqref{eq_corona} is a necessary condition for $L^1$ exact controllability, and the equivalence between \eqref{eq_corona} and Condition~\ref{assumption1-L1} of Theorem~\ref{main_result2} is a consequence of Item~\ref{prop:lienHQ1} of Proposition~\ref{prop:lienQH} and Proposition~\ref{prop:eq_Q_corona_thm}.
\end{proof}

The necessity of Condition~\ref{cdt_corona} of Proposition~\ref{prop_nece_bezout} is the easy part for the solvability of Bézout's identity over the Radon measure algebra. The harder part would be to prove that Condition~\ref{cdt_corona} is also a sufficient condition to have the solvability of Bézout's identity. A first step of such a proof would consist in using exactly the trick from \cite{Fuhrmann_corona}, which  reduces a corona matrix problem to the corresponding 1D version, which can be written as follows.

\begin{conjecture}
\label{conjecture1}
Let $k \in \mathbb{N}$ and $q_j \in  M(\mathbb{R}_-)$ for $j=1,\dots,k$. Assume the existence of a $c>0$ such that
\begin{equation}
\label{eq:lemm_gen_Y_Corona1}
\sum_{j=1}^k \abs{\widehat{q}_j(p)} \ge c>0, \quad p \in \mathbb{C}.
\end{equation}
Then there exists $p_j \in  M(\mathbb{R}_-)$ for $j=1,\dots,k$ satisfying the equation
\begin{equation*}
\sum_{j=1}^k \widehat{q}_j(p)\widehat{p}_j(p) =1, \quad p \in \mathbb{C}.
\end{equation*}
\end{conjecture}

\begin{remark}
In fact, due to the structure of the difference delay systems that we consider, i.e., with a finite number of delays, it would be sufficient to prove Conjecture~\ref{conjecture1} for the measures $q_j$, for all $j=1,\dots,k$, belonging to the sub-algebra of Radon measures $M(\mathbb{R}_-)$ finitely generated by the Dirac measures $(\delta_{\Lambda_i-\Lambda_N})_{i=1,\dots,N}$.
\end{remark}

Notice that a proof of Conjecture~\ref{conjecture1} would imply a positive answer to Conjecture~\ref{conj:HY-exact} in the case $q$ equal to one, yielding a sufficient and necessary criterion for the $L^1$ exact controllability.

\subsection{Sufficient exact controllability criterion for \texorpdfstring{$C^k$}{Ck} functions}
\label{subsec:Ck}

We close this section by providing a positive exact controllability result which relies on the sufficient condition for the resolution of the Bézout identity over the distributional algebra in dimension $d$ given in \cite{YAMAMOTO_Multi_Ring_2011}, i.e., we provide an exact controllability criterion for steering the origin to regular solutions along System~\eqref{syst_lin_avec_sortie}. More precisely, let $C^k(\mathbb{R}_+,\mathbb{R}^d)$ (respectively, $C^k(\mathbb{R},\mathbb{R}^d)$), with $k$ integer, denote the space of $k$ times continuously differentiable functions defined on $\mathbb{R}_+$ (respectively, $\mathbb{R}$).

We need the following definition of controllability to agree with the $C^k$ functions.
\begin{definition}
\label{def_reachability_Ck}
Let $k$ be an integer and set
\begin{equation*}
X_{k}^Q:=\left\{\,y \in  C^k(\mathbb{R}_+,\mathbb{R}^d)\,|\,\pi(Q*y)=0  \right\}.
\end{equation*}

System~\eqref{syst_lin_avec_sortie} is \emph{$C^k$ exactly controllable} if, for every $y \in X_{k}^Q$, there exists $u \in C^0(\mathbb{R},\mathbb{R}^m)$ compactly supported in $\mathbb{R}_-$ such that $\pi(A*u)=y$.
\end{definition}

\begin{remark}
Note that, contrarily to the notions of controllability that we introduced previously, $C^k$ exact controllability requires the control $u$ to belong only to the less regular space of continuous functions, instead of requiring it to belong to a space similar to that of the state $y$.
\end{remark}

The proposition below states a sufficient controllability criteria for the $C^k$ functions.

\begin{proposition}\label{prop-Ckexact}
If Condition~\ref{cdt_corona} of Proposition~\ref{prop_nece_bezout} holds true then
System~\eqref{syst_lin_avec_sortie} is $C^k$ exactly controllable for some integer $k$.
\end{proposition}

\begin{proof} In this proof, we fix $\vertii{\cdot}$ as the Euclidean norm. We proceed in two steps.

\begin{enumerate}[label={\textbf{Step~\arabic*.}}, ref={Step~\arabic*}, itemindent=*, leftmargin=0pt]

\item We aim to apply \cite[Theorem~5.1]{YAMAMOTO_Multi_Ring_2011}. Note first that, thanks to the classical P\'olya--Szeg\H{o} bound provided in \cite[Part Three, Problem~206.2]{Polya1998Problems}, there exists a positive integer $D$ such that the zeros of $\det \left(\widehat{Q}(\cdot)\right)$ have multiplicity at most $D$.

Assume that Condition~\ref{cdt_corona} of Proposition~\ref{prop_nece_bezout} holds true. Pick $p \in \mathbb{C}$ such that $\det \widehat{Q}(p)=0$ and denote by $\ell$ the dimension of the left kernel of $\widehat{Q}(p)$. Let $G_{p} \in \mathcal{M}_{\ell,d}(\mathbb{C})$ be such that its rows form an orthonormal basis of the left kernel of $ \widehat{Q}(p)$. Hence, $G_{p} \widehat{Q}(p)=0$ and $\vertii{z^T G_p} = \vertii{z}$ for every $z \in \mathbb C^\ell$.

We claim that the rows of the matrix $G_p \widehat{P}\in \mathcal{M}_{\ell,m}(\mathbb{C})$ are linearly independent. Indeed, if this were not the case, there would exist $z \in \mathbb{C}^{\ell}$ such that $\vertii{z} = 1$ and $z^T G_p\widehat{P}=0$. Thus we would have $z^T G_p \widehat{Q}(p)=0$, $z^T G_p\widehat{P}=0$, and $\vertii{z^T G_p}=1$, in contradiction with Condition~\ref{cdt_corona} of Proposition~\ref{prop_nece_bezout}. 
   
We deduce that $G_{p}\widehat{P}$ has right inverses and we use $(G_{p}\widehat{P})^{-1}$ to denote its Moore--Penrose right inverse, i.e., $(G_{p}\widehat{P})^{-1} = (G_{p}\widehat{P})^T (G_{p}\widehat{P} \widehat{P}^T G_p^T)^{-1}$. Define
\begin{equation*}
\widehat{\Phi}(p)=\left(G_{p}\widehat{P}\right)^{-1} G_{p}
\end{equation*}
and notice that $\widehat{\Phi}(p)$ satisfies
\begin{equation}
\label{eq:centrale2}
G_p\widehat{P}\widehat{\Phi}(p)=G_p.
\end{equation}

We now claim that there exists $c>0$ (independent of the zero $p$ of $\det \widehat{Q}$) such that
\begin{equation}
\label{eq:centrale3}
\vertiii{\widehat{\Phi}(p)} \le c.
\end{equation}
Indeed, arguing by contradiction yields a sequence $(p_n)_{n \in \mathbb{N}}$ of zeros of $\det \widehat Q$ such that $\vertiii{\left(G_{p_n} \widehat P\right)^{-1}}$ tends to infinity as $n \to +\infty$. We have $\vertiii{\left(G_{p_n} \widehat P\right)^{-1}} = 1/\lambda_{\min}(G_{p_n}\widehat{P} \widehat{P}^T G_{p_n}^T)$, where $\lambda_{\min}(\cdot)$ denotes the smallest eigenvalue of its argument. Let $(z_n)_{n \in \mathbb N}$ denote a sequence of vectors in $\mathbb C^\ell$ such that, for every $n \in \mathbb N$, $\vertii{z_n^T} = 1$ and $z_n$ is an eigenvector of $G_{p_n}\widehat{P} \widehat{P}^T G_{p_n}^T$ associated with its smallest eigenvalue. Then
\[
\vertii{z_n^T G_{p_n} \widehat P}^2 = \lambda_{\min}(G_{p_n}\widehat{P} \widehat{P}^T G_{p_n}^T) \underset{n \to +\infty}{\longrightarrow} 0,
\]
and, since $\vertii{z_n^T G_{p_n}} = 1$ for every $n \in \mathbb N$, and $G_{p_n} \widehat Q(p_n) = 0$, we obtain a contradiction with Condition~\ref{cdt_corona} in Proposition~\ref{prop_nece_bezout}, yielding \eqref{eq:centrale3}. 

The assumptions (18)\footnote{Equation~(18) in \cite{YAMAMOTO_Multi_Ring_2011} is the Laplace transform of Equation~(8) of the same article, which is an equation in $\left(\mathcal{E}'(\mathbb{R}) \right)^{d \times d} / \left( Q \right)$, where $(Q)$ is the ideal generated over $\mathcal{E}'(\mathbb{R}_-)$ by $Q$. Thus the equality in (18) has to be understood as an equality modulo the ideal generated by $\widehat Q(p)$, which is equivalent to requiring that $z^T \left(\widehat P(p) \widehat \Phi(p) - I_d\right) = 0$ for every $z^T$ in the left null space of $\widehat{Q}(p)$. Therefore, Equation~(18) in  \cite{YAMAMOTO_Multi_Ring_2011} is equivalent to our Equation~(73).} and (17) in the statement of \cite[Theorem~5.1]{YAMAMOTO_Multi_Ring_2011} are satisfied thanks to Equations \eqref{eq:centrale2} and \eqref{eq:centrale3}, respectively. We can then apply  \cite[Theorem~5.1]{YAMAMOTO_Multi_Ring_2011}, which ensures the existence of two distribution $R$ and $S$ with entries in $\mathcal{E}'(\mathbb{R}_{-})$ such that
\begin{equation*}
Q*R+P*S=\delta_0 I_d.
\end{equation*}

\item Since $S$ is a distribution with compact support, it has a finite order $k$. Let $y \in  X^{Q}_{k}$ and denote by $\tilde{y} \in C^k(\mathbb{R},\mathbb{R}^d)$ an extension on $\mathbb{R}$ of $y$ having a support bounded on the left. Set $\omega=S*Q*\tilde{y}$. Then $\omega$ belongs to $ C^0(\mathbb{R},\mathbb{R}^m)$ and it has a compact support included in $\mathbb{R}_-$. It follows from a similar argument given to obtain the equation \eqref{eq:sol-con-B} that $y=\pi\left(A * \omega\right)$, which proves the $C^k$ exact controllability. \qedhere
\end{enumerate}
\end{proof}

\begin{proof}[Proof of Theorem~\ref{thm:Ck}]
Notice that $y \in X_{k}^Q$ if and only if $[-\Lambda_N,0]\ni t\mapsto y(t+\Lambda_N)$
is $C^k$-admissible for System~\eqref{system_lin_formel2} in the sense of Definition~\ref{def:compatibility}. 
Recall that, by Proposition~\ref{Prop2:main_result}, Conditions \ref{assumption1-conject} and \ref{assumption2-conject} of Conjecture~\ref{conj:HY-exact} imply that Condition~\ref{cdt_corona} of Proposition~\ref{prop_nece_bezout} holds true.
Hence, by Proposition~\ref{prop-Ckexact}, every $C^k$-admissible function for System~\eqref{system_lin_formel2}
is in the range of $E_q(T)$ for some $T>0$ possibly depending on $y$. By Theorem~\ref{lem:RanE_con}, moreover, 
the space of  $C^k$-admissible functions for System~\eqref{system_lin_formel2} is contained in $\operatorname{Ran} E_q(d\Lambda_N)$.  
The conclusion then follows by Proposition~\ref{prop_var_const_formula} and the fact that
if $x_0$ is $C^k$-admissible for System~\eqref{system_lin_formel2} then 
$\Upsilon_q(t)x_0$ is  $C^k$-admissible for System~\eqref{system_lin_formel2} for every $t\ge 0$.
\end{proof}

\section{Applications}
\label{sec:applications}

In this section, we show in simple cases how to derive from Theorems~\ref{main_result1} and \ref{main_result2} Kalman-type conditions for controllability in the single input case $m=1$. By a Kalman-type condition, we refer to a  frequency-free  criterion for controllability. For instance, in the single-delay case $N=1$, approximate and exact controllability (in any fixed time $T \geq d \Lambda_1$) coincide and are equivalent to the classical \emph{Kalman rank condition} stating that the rank of the controllability matrix $[B,A_1B,\dots,A_1^{d-1}B]$ is equal to the state space dimension $d$. We first partially recover criteria given in \cite{Chitour2020Approximate} in the case of two delays, two space dimensions, and a single input, and then extend such a study to the case of two delays, three space dimensions, and a single input. 

Up to a time-rescaling, we can assume from now on that $(\Lambda_1,\Lambda_2)=(L,1)$ with $L \in (0,1)$. The case $L \in \mathbb Q$ was completely addressed in \cite{Chitour2020Approximate} and we assume for the rest of this section that $L$ is irrational.

\subsection{Two delays, two space dimensions, and a single input}
\label{sec:N=D=2,m=1}

In this section, we recover some results given in \cite[Theorem~4.1]{Chitour2020Approximate}, which  concern the particular case where $N=d=2$ and $m=1$.
The results in \cite{Chitour2020Approximate} are stated for matrices with complex coefficients, hence, what we actually recover here are some results of \cite[Theorem~4.1]{Chitour2020Approximate} restricted to the case of matrices with real coefficients. On the other hand, the results in \cite{Chitour2020Approximate} are stated only for $L^2$ controllability, but here we consider $L^q$ controllability for any $q \in [1, +\infty)$.

Fix then $A_1,A_2\in \mathcal{M}_{2,2}(\mathbb{R})$ and $B\in  \mathcal{M}_{2,1}(\mathbb{R})$. Given a matrix $A\in \mathcal{M}_{2,2}(\mathbb{R})$, we say that the pair $(A,B)$ is controllable if it satisfies the Kalman rank condition. We have three cases.
\begin{enumerate}[I)]
\item $\Ran A_2 \subset \Ran B$ or both pairs $(A_1,B)$, $(A_2,B)$ are not controllable. In both subcases, System~\eqref{system_lin_formel2} is not even $L^q$ approximately controllable. Indeed, in the first subcase, Condition~\ref{assumption2bis} of Theorems~\ref{main_result1} and~\ref{main_result2} is not satisfied. In the second subcase, one can assume with no loss of generality that 
\[
A_1=\begin{pmatrix}a_1&0\\a_2&a_3
\end{pmatrix}, \quad
A_2=\begin{pmatrix}b_1&0\\b_2&b_3
\end{pmatrix}, \quad
B=\begin{pmatrix}0\\1
\end{pmatrix},
\]
with $b_1\neq 0$ (otherwise we are back to the first subcase). Then the first coordinate $x_1$ of the state $x$ is not controllable since one has $x_1(t)=a_1x_1(t-L)+x_1(t-1)$.

\item $\Ran A_2  \not\subset \Ran B$ and exactly one of the pairs $(A_1,B)$, $(A_2,B)$ is controllable. Then System~\eqref{system_lin_formel2} is $L^q$ approximately controllable in time $2 \Lambda_2$. Indeed, notice first that $\rank[A_2,B] = 2$, and we are thus left to show that Condition~\ref{assumption1bis} of Theorem~\ref{main_result1} hold true, which is equivalent to proving that Condition~\ref{app-b} of Proposition~\ref{Prop2:main_result0} holds true. Assume, for instance, that $(A_1,B)$ is controllable and $(A_2,B)$ is not. Hence, up to a linear change of variables, we can assume with no loss of generality that
\[
A_1=\begin{pmatrix}a_1&a_2\\a_3&a_4
\end{pmatrix}, \quad
A_2=\begin{pmatrix}b_1&0\\b_2&b_3
\end{pmatrix}, \quad
B=\begin{pmatrix}0\\1
\end{pmatrix},
\]
with $a_2 \neq 0$ and $b_1 \neq 0$. For $g \in \mathbb C^2$ with $\vertii{g} = 1$, one has either $g^T B \neq 0$ or $g^T = (\alpha, 0)$ for some $\alpha \in \mathbb C$ with $\abs{\alpha} = 1$. Hence, in order to show Condition~\ref{app-b} of Proposition~\ref{Prop2:main_result0}, it suffices to show that $\vertii{g^T H(p)}$ is nonzero for every $p \in \mathbb C$ and with $g^T = (1, 0)$. For every $p \in \mathbb C$, one checks that the second coordinate of $g^T H(p)$ is equal to $-a_2 e^{-p L}$, which never vanishes, yielding the conclusion. The case where $(A_1,B)$ is not controllable and $(A_2,B)$ is can be handled similarly.

\item\label{C3D2} $(A_1,B)$ and $(A_2,B)$ are both controllable. Let $B^{\bot} \in \mathbb{R}^2$ be the unique vector such that $\det(B,B^{\bot})=1$ and $B^T B^{\bot}=0$. Set
\[
\beta= \frac{\det \left( \left[B,A_2B \right] \right)}{\det \left( \left[B,A_1B \right] \right)}, \qquad \alpha= \det  \left( [B,(A_2-\beta A_1)B^{\bot}] \right).
\]
Up to a linear change of coordinates, we can assume that
\begin{equation}
\label{eq:new_coordinates-221-case3}
A_1=\begin{pmatrix}
0 & 1  \\
a_1 & a_2 
\end{pmatrix}, \quad A_2=\begin{pmatrix}
\alpha & \beta  \\
b_1 & b_2 
\end{pmatrix},  \quad B=\begin{pmatrix}
0   \\
1 
\end{pmatrix}.
\end{equation}

The holomorphic map $H$ is now given by
\begin{equation}
\label{eq_221}
H(p)=I_2- e^{-pL} A_1 - e^{- p}A_2=\begin{pmatrix}
1-\alpha e^{-p} & -e^{-pL}-\beta e^{-p}  \\
\ast & \ast
\end{pmatrix},\qquad p \in \mathbb{C}.
\end{equation}
Let $\alpha=|\alpha|e^{i \theta}$ with $\theta \in \{0,\pi\}$. Since $\beta \neq 0$, we have that $\rank[A_2,B]=2$. Then, by Theorem~\ref{main_result1} and Equation~\eqref{eq_221}, we have that System~\eqref{system_lin_formel2} is not $L^q$ approximately controllable in time $2\Lambda_2$ if and only if
\begin{align*}
\exists p\in\mathbb{C}\mbox{ s.t.\ $\rank\left[{H}(p),B\right]<2$} &\iff \exists p\in\mathbb{C}\mbox{ s.t.\ $1-\alpha e^{-p}=0$ and  $e^{-pL}+\beta e^{-p}=0$}\\
&\iff 0 \in S,
\end{align*}
where $S = \left\{ \beta+ |\alpha|^{1-L}e^{i (\theta+2k \pi)(1-L)}\suchthat k \in \mathbb{Z} \right\}$. Since $L$ is irrational, notice that $\overline S$, is the circle in $\mathbb C$ of center $\beta$ and radius $|\alpha|^{1-L}$, denoted hereafter by $C$. By Theorem~\ref{main_result2} and Equation~\eqref{eq_221}, one can prove in the same way that System~\eqref{system_lin_formel2} is not $L^1$ exactly controllable in time $2 \Lambda_2$ if $0 \in C$.
\end{enumerate}

In the third case above, the fact that there exists $p\in\mathbb{C}$ so that $1-\alpha e^{-p}=0$ and  $e^{-pL}+\beta e^{-p}=0$ can be equivalently written as
\begin{equation*}
MY(p)=\begin{pmatrix}1\\0\end{pmatrix},\qquad 
\hbox{ where }
M=\begin{pmatrix}0&\alpha\\1&\beta\end{pmatrix}
\hbox{ and } Y(p)=\begin{pmatrix}e^{-pL}\\e^{-p}\end{pmatrix}.
\end{equation*}

The above computations are a particular case of a more general situation, described next.

\begin{lemma}
\label{lem:M-Y-Z}
Let $M$ be a $2 \times 2$ matrix with complex coefficients, $Z \in \mathbb C^2 \setminus \{0\}$, $L \in (0, 1)\setminus \mathbb{Q}$, and define the map $Y: p \mapsto (e^{-p L}, e^{-p})^T$.
\begin{enumerate}[label={\roman*)}, ref={\roman*}]
\item If $M$ is not invertible, then the following assertions are equivalent:
\begin{enumerate*}[(a)]
\item $Z \in \Ran M$;
\item $Z \in M Y(\mathbb C)$;
\item $Z \in M \overline{Y(\mathbb C)}$.
\end{enumerate*}

\item If $M$ is invertible, set $M^{-1} Z = (a_1, a_2)^T$ and define
\[
S = \left\{a_1 - \abs{a_2}^L e^{i L (\theta_2 + 2 k \pi)} \suchthat k \in \mathbb Z\right\},
\]
where $a_2 = \abs{a_2} e^{i \theta_2}$ and $\theta_2 \in \mathbb R$. Then the following assertions hold true:
\begin{enumerate*}[label={\theenumi-\roman*)}]
\item $Z \in \allowbreak M Y(\mathbb C)$ if and only if $0 \in S$;
\item $Z \in M \overline{Y(\mathbb C)}$ if and only if $0 \in C$, the latter being the circle in $\mathbb C$ of center $a_1$ and radius $\abs{a_2}^L$.
\end{enumerate*}
\end{enumerate}
\end{lemma}

Our two main results allowed us to recover partially, in the case of matrices with real coefficients, the algebraic controllability result stated in \cite[Theorem~4.1]{Chitour2020Approximate} when $N=d=2$ and $m=1$, which consider $L^2$ (approximate and exact) controllability concepts while in the present paper, we are instead dealing with $L^q$ approximate controllability and $L^1$ exact controllability. We have shown that three possibilities occur: \begin{enumerate*}[I)] \item neither approximate nor exact controllability hold true; \item $L^q$ approximate controllability holds true but we cannot say anything about  exact controllability; \item there exists a countable subset $S$ of the complex plane, completely characterized by $A_1,A_2,B$, and $L$, which is dense in a circle $C$, such that $L^q$ approximate controllability is equivalent to the fact that $0\notin S$ and a necessary condition for $L^1$ exact controllability is $0\notin C$.\end{enumerate*}

\begin{remark}
For Case~\ref{C3D2}, we can assume with no loss of generality (see \cite[Lemma~4.5]{Chitour2020Approximate}) that $a_1=a_2=b_1=b_2=0$ in \eqref{eq:new_coordinates-221-case3}, so that Theorem~\ref{prop_fond_exact_controllability_L1} implies that System~\eqref{system_lin_formel2} is $L^1$ exactly controllable if and only if the Bézout identity \eqref{eq:exact_cont0} is solvable in $M(\mathbb R_-)$, that is, if and only if there exists $(s_1,s_2,r_1,r_2,r_3,r_4) \in M(\mathbb{R}_-)^6$ such that
\begin{equation}
\label{cas_particulier_bezout_main}
\left\{
\begin{aligned}
 & (\delta_{-1} - \alpha \delta_0) * r_1 - (\delta_{L - 1} + \beta \delta_0) * r_2 = \delta_0, \\
 & (\delta_{-1} - \alpha \delta_0) * r_3 - (\delta_{L - 1} + \beta \delta_0) * r_4 = 0, \\
 & \delta_{-1} * r_2 + s_1 = 0, \\
 & \delta_{-1} * r_4 + s_2 = \delta_0.\\
\end{aligned}
\right.
\end{equation}
The second and fourth equations of \eqref{cas_particulier_bezout_main} are satisfied if one chooses
\[
r_3 = \delta_{L - 1} + \beta \delta_0, \quad
r_4 = \delta_{-1} - \alpha \delta_0, \quad
s_2 = \delta_0 + \alpha \delta_{-1} - \delta_{-2}.
\]
In addition, as soon as the first equation of \eqref{cas_particulier_bezout_main} is satisfied, the third one can be satisfied by setting $s_1 = - \delta_{-1} * r_2$, and we hence focus in the first equation of \eqref{cas_particulier_bezout_main} in the sequel, i.e.,
\begin{equation}
\label{cas_particulier_bezout_main2}
(\delta_{-1} - \alpha \delta_0) * r_1 - (\delta_{L - 1} + \beta \delta_0) * r_2 = \delta_0.
\end{equation}

Let $q_1 = \delta_{-1} - \alpha \delta_0$ and $q_2 = -\delta_{L - 1} - \beta \delta_0$, and notice that their Laplace transforms are given by $\widehat q_1(p) = e^p - \alpha$ and $\widehat q_2(p) = -e^{(1-L)p} - \beta$ for $p \in \mathbb C$. By definition of $S$, we have that $0 \in S$ if and only if there exists $p \in \mathbb C$ such that $\widehat q_1(p) = \widehat q_2(p) = 0$, and one can show that $0 \in \overline S$ if and only if $\inf_{p \in \mathbb C} \abs{\widehat q_1(p)} + \abs{\widehat q_2(p)} = 0$. Hence, the condition $0 \notin \overline S$ is equivalent to the existence of a constant $c > 0$ such that $\abs{\widehat q_1(p)} + \abs{\widehat q_2(p)} \geq c$ for every $p \in \mathbb C$. If Conjecture~\ref{conjecture1} held true, then it would imply that there exist $r_1$ and $r_2$ in $M(\mathbb R_-)$ such that their Laplace transforms satisfy $\widehat q_1(p) \widehat r_1(p) + \widehat q_2(p) \widehat r_2(p) = 1$, which is exactly the Laplace transform of \eqref{cas_particulier_bezout_main2}. 
In particular, this shows that, if $0 \notin \overline S$ and Conjecture~\ref{conjecture1} holds true, then System~\eqref{system_lin_formel2} in the case of the present example is $L^1$ exact controllable.

Note that the condition $0 \notin \overline S$ is shown in \cite{Chitour2020Approximate} to be a necessary and sufficient condition for the $L^2$ exact controllability of System~\eqref{system_lin_formel2} in Case~\ref{C3D2}. Unfortunately, even for this simple example, it is not clear if the $L^2$ exact controllability is equivalent to the $L^1$ exact controllability.
\end{remark}

\subsection{Two delays, three space dimensions, and a single input: a geometric locus controllability result}
\label{subsec:2delay3dim1in}

In this section, we assume that $N=2$, $d=3$, and $m=1$. In this case $A_1,A_2\in \mathcal{M}_{3,3}(\mathbb{R})$ and $B\in  \mathcal{M}_{3,1}(\mathbb{R})$. Up to a linear change of variables, we assume in the sequel that $B=(0,0,1)^T$. We use $r_0$ to denote the dimension of the vector space spanned by the three vectors $B,A_1B,A_2B$. 

We subdivide the discussion in three cases. 
\begin{enumerate}[I)] 
\item $\rank[A_2,B]\leq 2$ or $r_0=1$. Then, in both subcases, System~\eqref{system_lin_formel2} is not even approximately controllable. Indeed, in the first subcase, the rank condition 
\ref{assumption2bis} of Theorem~\ref{main_result1} is violated. 
In the second subcase, the matrices $A_1$ and $A_2$ are of the form
\[
A_1=\begin{pmatrix}
\widetilde{A}_1 &0_{2,1}\\
\ast &a_1
\end{pmatrix},
\quad A_2=\begin{pmatrix}
\widetilde{A}_2&0_{2,1}\\
\ast &a_2
\end{pmatrix}.
\]
Then the subsystem of System~\eqref{system_lin_formel2}  made of the first two coordinates is uncoupled to the third coordinate and is uncontrolled. The overall system is then 
not controllable.

\item\label{dim3-case2} $\rank[A_2,B]=3$ and $r_0=2$. 
We only treat the case where $A_2B$ is not colinear to $B$ (the other case being entirely similar). Up to a linear change of variables (in a basis having $A_2B$ as second vector and $B$ as third one), one can transform the matrices $A_1$ and $A_2$ into the form 
\[
A_1=\begin{pmatrix}
\widetilde{A}_1 &\begin{matrix}0\\ \alpha\end{matrix}\\
\ast &\ast
\end{pmatrix},
\quad A_2=\begin{pmatrix}
\widetilde{A}_2 &\begin{matrix}0\\ 1\end{matrix}\\
\ast &0
\end{pmatrix},
\]
where the matrices $\widetilde{A}_1$ and  $\widetilde{A}_2$ are $2\times 2$.

The holomorphic map $H$ associated with System~\eqref{system_lin_formel2} is equal to 
\begin{equation}
\label{eq_231}
H(p)=\begin{pmatrix}
I_2- e^{-pL} \widetilde{A}_1 - e^{- p}\widetilde{A}_2&
\begin{matrix}0\\-(e^{-pL}\alpha+e^{-p})
\end{matrix}\\
\ast & \ast
\end{pmatrix}.
\end{equation}
Then, by Theorem~\ref{main_result1} and Equation~\eqref{eq_231}, System~\eqref{system_lin_formel2} is not $L^q$ approximately controllable in time $3\Lambda_2$ if and only if there exists $p\in\mathbb{C}$ such that 
$\rank\left[{H}(p),B\right]<3$, i.e., either \begin{enumerate*}[(a)] \item\label{dim3-case2a} $e^{-pL}\alpha+e^{-p}=0$ and 
$\det(I_2- e^{-pL} \widetilde{A}_1 - e^{- p}\widetilde{A}_2)=0$ or \item\label{dim3-case2b} $e^{-pL}\alpha+e^{-p}\neq 0$ and the $(1,1)$ and $(1,2)$ coefficients of the $3\times 3$ matrix $H(p)$ are both equal to zero.\end{enumerate*}

In the subcase \ref{dim3-case2a}, the second condition is equivalent to the fact that $\det\big(e^{pL}I_2-(\widetilde{A}_1 - \alpha\widetilde{A}_2)\big)=0$, i.e., $e^{pL}$ is one of the eigenvalues of $\widetilde{A}_1 - \alpha\widetilde{A}_2$. Hence we are in the situation of Lemma~\ref{lem:M-Y-Z} with $M = \begin{pmatrix}\alpha & 1 \\ 1 & 0\end{pmatrix}$ and at most two vectors $Z$ corresponding to the nonzero eigenvalues of $\widetilde{A}_1 - \alpha\widetilde{A}_2$. It follows that there exist at most two countable sets of complex numbers $S_1, S_2$, completely characterized in terms of $L$ and the coefficients of $A_1, A_2$, which are dense in two circles $C_1, C_2$, respectively, and such that System~\eqref{system_lin_formel2} is $L^q$ approximately controllable if and only if $0$ does not belongs to the union of $S_1$  and $S_2$, while System~\eqref{system_lin_formel2} is not $L^1$ exactly controllable if $0$ belongs to the union of $C_1$  and $C_2$. (Note that each $S_i$ can be equal to $C_i$ if the latter reduces to a point.) It is not difficult to see that the subcase \ref{dim3-case2b} also boils down to a similar situation, but with at most one circle.

\item $\rank[A_2,B]=3$ and $r_0=3$. 
Up to a linear change of variables (in a basis having $A_1B$ as first vector, $A_2B$ as second vector, and $B$ as third one), one can transform the matrices $A_1$ and $A_2$ into the form 
\[
A_1=\begin{pmatrix}
\widetilde{A}_1 &\begin{matrix}1\\ 0\end{matrix}\\
\ast &0
\end{pmatrix},
\quad A_2=\begin{pmatrix}
\widetilde{A}_2 &\begin{matrix}0\\ 1\end{matrix}\\
\ast &0
\end{pmatrix},
\]
where the matrices $\widetilde{A}_1$ and  $\widetilde{A}_2$ are $2\times 2$. 

The holomorphic map $H$ associated with System~\eqref{system_lin_formel2} is equal to 
\begin{equation*}
H(p)=\begin{pmatrix}
I_2- e^{-pL} \widetilde{A}_1 - e^{- p}\widetilde{A}_2&
\begin{matrix}-e^{-pL}\\-e^{-p}
\end{matrix}\\
\ast&0
\end{pmatrix}.
\end{equation*}

We denote by $v_1$ and $v_2$ the columns of the matrix $\widetilde{H} (p)=I_2-e^{-pL} \widetilde{A}_{1}-e^{-p} \widetilde{A}_{2}$ and we set $v_3=(e^{-pL}\ e^{-p})^T$. Note that the rank assumption on $[A_2,B]$ implies that the first row of $\widetilde A_2$ is not equal to zero. 

It is immediate to see that there exists $p\in\mathbb{C}$ so that $\rank\left[{H}(p),B\right]<3$ if and only if both $v_1$ and $v_2$ are colinear to the nonzero vector $v_3$, i.e., $\det(v_1,v_3)=0$ and $\det(v_2,v_3)=0$. These conditions can be rewritten as two scalar equations $v_3^T Q_1 v_3+e^{-p} = 0$ and $v_3^TQ_2v_3+e^{-p L}=0$, where $Q_1,Q_2$ are symmetric matrices with real coefficients and at least one between $Q_1$ and $Q_2$ is not equal to zero because of the rank assumption on $[A_2,B]$.

The two previous equations in the unknowns $e^{-pL}$ and $e^{-p}$ define two distinct conic sections, at least one of them being nontrivial. Therefore, they have therefore $k$ distinct intersection points with $0 \le k \le 4$. For each of them, one can completely characterize, in terms of $L$ and the coefficient of $A_1,A_2$, a countable set of complex numbers $S_j$, $1\leq j\leq k$, which is dense in a circle $C_j$, such that  System~\eqref{system_lin_formel2} is $L^q$ approximately controllable if and only if $0$ does not belongs to the union of the $S_j$, while System~\eqref{system_lin_formel2} is not $L^1$ exactly controllable if $0$ does belongs to the union of the $C_j$.
\end{enumerate}

In conclusion, if $N=2$, $d=3$, and $m=1$, Theorem~\ref{main_result1} (respectively, Theorem~\ref{main_result2}) allows one to derive frequency-free necessary and sufficient (respectively, necessary) criteria for $L^q$ approximate (respectively, $L^1$ exact) controllability, and the results are qualitatively similar to those obtained if $N=d=2$ and $m=1$, with the difference in Case~\ref{C3D2} where now we may have up to four distinct countable sets $S_j$, $1\leq j\leq k$, each of them dense in a circle $C_j$.

\bibliographystyle{abbrv} 
\bibliography{biblio_DCDS}

\begin{thebibliography}{10}

\bibitem{Avellar}
C.~E. Avellar and J.~K. Hale.
\newblock On the zeros of exponential polynomials.
\newblock {\em J. Math. Anal. Appl.}, 73(2):434--452, 1980.

\bibitem{baratchart}
L.~Baratchart, S.~Fueyo, G.~Lebeau, and J.-B. Pomet.
\newblock Sufficient stability conditions for time-varying networks of
  telegrapher's equations or difference-delay equations.
\newblock {\em SIAM J. Math. Anal.}, 53(2):1831--1856, 2021.

\bibitem{bastin2016stability}
G.~Bastin and J.-M. Coron.
\newblock {\em Stability and boundary stabilization of 1-{D} hyperbolic
  systems}, volume~88 of {\em Progress in Nonlinear Differential Equations and
  their Applications}.
\newblock Birkh\"{a}user/Springer, [Cham], 2016.
\newblock Subseries in Control.

\bibitem{bony2001cours}
J.-M. Bony.
\newblock {\em Cours d'analyse: th{\'e}orie des distributions et analyse de
  Fourier}.
\newblock Editions Ecole Polytechnique, 2001.

\bibitem{carleson1962interpolations}
L.~Carleson.
\newblock Interpolations by bounded analytic functions and the corona problem.
\newblock {\em Annals of Mathematics}, pages 547--559, 1962.

\bibitem{Chitour2016Stability}
Y.~Chitour, G.~Mazanti, and M.~Sigalotti.
\newblock Stability of non-autonomous difference equations with applications to
  transport and wave propagation on networks.
\newblock {\em Netw. Heterog. Media}, 11(4):563--601, 2016.

\bibitem{Chitour2020Approximate}
Y.~Chitour, G.~Mazanti, and M.~Sigalotti.
\newblock Approximate and exact controllability of linear difference equations.
\newblock {\em J. \'{E}c. polytech. Math.}, 7:93--142, 2020.

\bibitem{coron}
J.-M. Coron.
\newblock {\em Control and nonlinearity}, volume 136 of {\em Mathematical
  Surveys and Monographs}.
\newblock American Mathematical Society, Providence, RI, 2007.

\bibitem{CoronOptimal}
J.-M. Coron and H.-M. Nguyen.
\newblock On the optimal controllability time for linear hyperbolic systems
  with time-dependent coefficients.
\newblock Preprint arXiv:2103.02653.

\bibitem{CoNg}
J.-M. Coron and H.-M. Nguyen.
\newblock Dissipative boundary conditions for nonlinear 1-{D} hyperbolic
  systems: sharp conditions through an approach via time-delay systems.
\newblock {\em SIAM J. Math. Anal.}, 47(3):2220--2240, 2015.

\bibitem{Coron2019Optimal}
J.-M. Coron and H.-M. Nguyen.
\newblock Optimal time for the controllability of linear hyperbolic systems in
  one-dimensional space.
\newblock {\em SIAM J. Control Optim.}, 57(2):1127--1156, 2019.

\bibitem{Coron2021Null}
J.-M. Coron and H.-M. Nguyen.
\newblock Null-controllability of linear hyperbolic systems in one dimensional
  space.
\newblock {\em Systems Control Lett.}, 148:Paper No. 104851, 8, 2021.

\bibitem{Fuhrmann_corona}
P.~A. Fuhrmann.
\newblock On the corona theorem and its application to spectral problems in
  {H}ilbert space.
\newblock {\em Trans. Amer. Math. Soc.}, 132:55--66, 1968.

\bibitem{Hale}
J.~K. Hale and S.~M. Verduyn~Lunel.
\newblock {\em Introduction to functional-differential equations}, volume~99 of
  {\em Applied Mathematical Sciences}.
\newblock Springer-Verlag, New York, 1993.

\bibitem{hale2002stabilization}
J.~K. Hale and S.~M. Verduyn~Lunel.
\newblock Strong stabilization of neutral functional differential equations.
\newblock {\em IMA J. Math. Control Inform.}, 19(1-2):5--23, 2002.
\newblock Special issue on analysis and design of delay and propagation
  systems.

\bibitem{Henry}
D.~Henry.
\newblock Linear autonomous neutral functional differential equations.
\newblock {\em J. Differential Equations}, 15:106--128, 1974.

\bibitem{Jacobs}
M.~Q. Jacobs and C.~E. Langenhop.
\newblock Criteria for function space controllability of linear neutral
  systems.
\newblock {\em SIAM J. Control Optim.}, 14(6):1009--1048, 1976.

\bibitem{kalman1972realization}
R.~Kalman and M.~Hautus.
\newblock Realization of continuous-time linear dynamical systems: rigorous
  theory in the style of schwartz.
\newblock In {\em Ordinary Differential Equations}, pages 151--164. Elsevier,
  1972.

\bibitem{kamen1976module}
E.~W. Kamen.
\newblock Module structure of infinite-dimensional systems with applications to
  controllability.
\newblock {\em SIAM Journal on Control and Optimization}, 14(3):389--408, 1976.

\bibitem{maggi2012sets}
F.~Maggi.
\newblock {\em Sets of finite perimeter and geometric variational problems: an
  introduction to Geometric Measure Theory}.
\newblock Number 135 in Cambridge Studies in Advanced Mathematics. Cambridge
  University Press, 2012.

\bibitem{Manitius}
A.~Manitius and R.~Triggiani.
\newblock Function space controllability of linear retarded systems: a
  derivation from abstract operator conditions.
\newblock {\em SIAM J. Control Optim.}, 16(4):599--645, 1978.

\bibitem{Mazanti2017Relative}
G.~Mazanti.
\newblock Relative controllability of linear difference equations.
\newblock {\em SIAM J. Control Optim.}, 55(5):3132--3153, 2017.

\bibitem{Miller2004}
L.~Miller.
\newblock Controllability cost of conservative systems: resolvent condition and
  transmutation.
\newblock {\em J. Funct. Anal.}, 218(2):425--444, 2005.

\bibitem{connor}
D.~A. O'Connor and T.~J. Tarn.
\newblock On the function space controllability of linear neutral systems.
\newblock {\em SIAM J. Control Optim.}, 21(2):306--329, 1983.

\bibitem{polderman1998introduction}
J.~W. Polderman and J.~C. Willems.
\newblock {\em Introduction to mathematical systems theory}, volume~26 of {\em
  Texts in Applied Mathematics}.
\newblock Springer-Verlag, New York, 1998.
\newblock A behavioral approach.

\bibitem{Polya1998Problems}
G.~P\'{o}lya and G.~Szeg\H{o}.
\newblock {\em Problems and theorems in analysis. {I}}.
\newblock Classics in Mathematics. Springer-Verlag, Berlin, 1998.
\newblock Series, integral calculus, theory of functions, Translated from the
  German by Dorothee Aeppli, Reprint of the 1978 English translation.

\bibitem{Rudin1987Real}
W.~Rudin.
\newblock {\em Real and complex analysis}.
\newblock McGraw-Hill Book Co., New York, third edition, 1987.

\bibitem{rudin1991functional}
W.~Rudin.
\newblock {\em Functional Analysis}.
\newblock International series in pure and applied mathematics. McGraw-Hill,
  1991.

\bibitem{salamon1984control}
D.~Salamon.
\newblock {\em Control and observation of neutral systems}, volume~91 of {\em
  Research Notes in Mathematics}.
\newblock Pitman (Advanced Publishing Program), Boston, MA, 1984.

\bibitem{schwartz1966theorie}
L.~Schwartz.
\newblock {\em Th\'{e}orie des distributions}.
\newblock Publications de l'Institut de Math\'{e}matique de l'Universit\'{e} de
  Strasbourg, IX-X. Hermann, Paris, 1966.
\newblock Nouvelle \'{e}dition, enti\'{e}rement corrig\'{e}e, refondue et
  augment\'{e}e.

\bibitem{sontag}
E.~D. Sontag.
\newblock {\em Mathematical control theory}, volume~6 of {\em Texts in Applied
  Mathematics}.
\newblock Springer-Verlag, New York, second edition, 1998.
\newblock Deterministic finite-dimensional systems.

\bibitem{Yutaka_Yamamoto}
Y.~Yamamoto.
\newblock {\em Realization Theory of infinite-dimensional linear systems}.
\newblock {PhD} thesis, University of Florida, Florida, United States, 1978.

\bibitem{yamamoto1981realization}
Y.~Yamamoto.
\newblock Realization theory of infinite-dimensional linear systems. {P}arts
  {I} and {II}.
\newblock {\em Math. Systems Theory}, 15:55--77 and 169--190, 1981/82.

\bibitem{YamamotoRealization}
Y.~Yamamoto.
\newblock Pseudo-rational input/output maps and their realizations: a
  fractional representation approach to infinite-dimensional systems.
\newblock {\em SIAM J. Control Optim.}, 26(6):1415--1430, 1988.

\bibitem{yamamoto1989reachability}
Y.~Yamamoto.
\newblock Reachability of a class of infinite-dimensional linear systems: an
  external approach with applications to general neutral systems.
\newblock {\em SIAM J. Control Optim.}, 27(1):217--234, 1989.

\bibitem{Yamamoto_coprimness_measure}
Y.~Yamamoto.
\newblock Coprimeness in the ring of psedorational transfer functions.
\newblock In {\em 2007 Mediterranean Conference on Control Automation}, pages
  1--6, 2007.

\bibitem{YAMAMOTO_Multi_Ring_2011}
Y.~Yamamoto.
\newblock Bézout identity over a ring of distributions—multivariable case.
\newblock {\em IFAC Proceedings Volumes}, 44(1):10098--10104, 2011.
\newblock 18th IFAC World Congress.

\bibitem{Yamamoto_Willems}
Y.~Yamamoto and J.~C. Willems.
\newblock Behavioral controllability and coprimeness for a class of
  infinite-dimensional systems.
\newblock In {\em Proceeding of the 47th IEEE Conference on Decision and
  Control}, pages 1513--1518, 2008.

\end{thebibliography}

\end{document}